\documentclass[aihp]{imsart}

\RequirePackage{amsthm,amsmath,amsfonts,amssymb}
\RequirePackage[numbers]{natbib}
\RequirePackage[colorlinks,citecolor=blue,urlcolor=blue]{hyperref}
\RequirePackage{graphicx}

\startlocaldefs
\numberwithin{equation}{section}
\theoremstyle{plain}
\newtheorem{theorem}{Theorem}[section]
\newtheorem{lemma}[theorem]{Lemma}
\newtheorem{proposition}[theorem]{Proposition}
\newtheorem{corollary}[theorem]{Corollary}

\theoremstyle{remark}
\newtheorem{remark}[theorem]{{Remark}}
\newtheorem{definition}[theorem]{Definition}

\newcommand{\supp}{\operatorname{supp}}
\newcommand{\diam}{\operatorname{diam}}

\newcommand{\textd}{{\rm d}\mkern0.5mu}
\newcommand{\texte}{{\rm  e}\mkern0.7mu}

\newcommand{\Var}{\text{\rm Var}}

\newcommand{\LL}{\mathcal L}
\newcommand{\MM}{\mathcal M}
\newcommand{\NN}{\mathcal N}

\newcommand{\RR}{\mathcal R}

\newcommand{\cR}{\RR}
\newcommand{\cL}{\LL}

\newcommand{\E}{\mathbb E}
\newcommand{\bbE}{\E}

\newcommand{\N}{\mathbb N}

\newcommand{\BbbP}{\mathbb P}
\newcommand{\bbP}{\BbbP}
\newcommand{\Q}{\mathbb Q}
\newcommand{\R}{\mathbb R}
\newcommand{\bbR}{\R}

\newcommand{\Z}{\mathbb Z}
\newcommand{\bbZ}{\Z}

\newcommand{\twoeqref}[2]{(\ref{#1}--\ref{#2})}

\newcommand{\cc}{{{\rm c}}}

\newcommand{\fraka}{\mathfrak a}

\newcommand{\wh}{\widehat}
\newcommand{\wt}{\widetilde}
\newcommand{\ol}{\overline}

\newcommand{\laweq}{\,\overset{\text{\rm law}}=\,}

\newcommand{\leb}{{\rm Leb}}
\newcommand{\Cov}{\text{\rm Cov}}

\newcommand{\Lawarrow}{{\,\overset{\text{\rm law}}\longrightarrow\,}}

\newcommand{\rmd}{\textd}
\newcommand{\rme}{\texte}

\newcommand{\fU}{\mathfrak D}
\newcommand{\abs}[1]{|#1|}

\newcommand{\la}{\langle}
\newcommand{\ra}{\rangle}
\newcommand{\ep}{\epsilon}

\newcommand{\Ind}{\text{\sf 1}}
\newcommand{\Leb}{\leb}
\newcommand{\bbN}{\N}
\newcommand{\bbQ}{\Q}

\def\myffrac#1#2 in #3{\raise 2.6pt\hbox{$#3 #1$}\mkern-1.5mu\raise 0.8pt\hbox{$#3/$}\mkern-1.1mu\lower 1.5pt\hbox{$#3 #2$}}
\newcommand{\ffrac}[2]{\mathchoice%
	{\myffrac{#1}{#2} in \scriptstyle}
	{\myffrac{#1}{#2} in \scriptstyle}
	{\myffrac{#1}{#2} in \scriptscriptstyle}
	{\myffrac{#1}{#2} in \scriptscriptstyle}
}

\newcommand\independent{\protect\mathpalette{\protect\independenT}{\perp}}
\def\independenT#1#2{\mathrel{\rlap{$#1#2$}\mkern3mu{#1#2}}}

\endlocaldefs

\begin{document}

\begin{frontmatter}

\title{
	Near-maxima of the two-dimensional\\Discrete Gaussian Free Field}
\runtitle{Near-maxima of  two-dimensional  DGFF}

\begin{aug}
\author[A]{\inits{MB}\fnms{Marek}~\snm{Biskup}\ead[label=e1]{biskup@math.ucla.edu}},
\author[B,C]{\inits{SG}\fnms{Stephan}~\snm{Gufler}\ead[label=e2]{gufler@math.uni-frankfurt.de}}
\and
\author[B]{\inits{OL}\fnms{Oren}~\snm{Louidor}\ead[label=e3]{oren.louidor@gmail.com}}
\address[A]{Department of Mathematics, UCLA, Los Angeles, California, USA\printead[presep={,\ }]{e1}}

\address[B]{Department of Data and Decision Sciences, Technion, Haifa, Israel\printead[presep={,\ }]{e3}}

\address[C]{Institute of Mathematics, Goethe University, Frankfurt am Main, Germany\printead[presep={,\ }]{e2}}
\end{aug}

\begin{abstract}
We consider the Discrete Gaussian Free Field (DGFF) in  domains~$D_N\subseteq\mathbb Z^2$ arising, via scaling by~$N$, from nice domains~$D\subseteq\mathbb R^2$. We study the statistics of the  values  order $\sqrt{\log N}$ below the absolute maximum. 
Encoded as a point process on~$D\times\mathbb R$,  the scaled  spatial distribution of these  near-extremal  level sets in $D_N$ and the field values (in units of~$\sqrt{\log N}$  below  the absolute maximum) tends, as~$N\to\infty$, in law 
to the product of the critical Liouville Quantum Gravity (cLQG)~$Z^D$ and the Rayleigh law. The convergence holds jointly with the extremal process, for which~$Z^D$ enters as the intensity measure of the limiting Poisson point process, and that of the DGFF itself; the~cLQG defined by the limit field then coincides with~$Z^D$. While the limit near-extremal process is measurable with respect to the limit continuum GFF, the limit extremal process is not.
Our results explain why the various ways to ``norm'' the lattice cLQG measure lead to the same limit object, modulo overall normalization.
\end{abstract}


\begin{keyword}[class=MSC]
\kwd{60G70}
\kwd{60G60}
\kwd{60G15}
\kwd{60G57}
\end{keyword}

\begin{keyword}
\kwd{Gaussian Free Field}
\kwd{log correlated fields}
\kwd{Liouville quantum gravity}
\kwd{extreme value theory}
\end{keyword}

\end{frontmatter}


\section{Introduction}
\noindent
Recent years have witnessed considerable progress in the understanding of extremal behavior of logarithmically correlated spatial random processes. A canonical example among these is the two-dimensional Discrete Gaussian Free Field (DGFF). In its $\Z^2$-based version, this is a 
centered Gaussian process $\{h^V_x\colon x\in V\}$ indexed by vertices in a (non-empty) set $V\subsetneq\Z^2$ with covariances
\begin{equation}
	\label{E:1.1}
	\Cov(h^V_x,h^V_y)=G^V(x,y),\quad x,y\in V,
\end{equation}
where $G^V$ is the Green function of the simple random walk. 
In  the normalization that we use,  $G^V(x,y)$ is the expected number of visits to~$y$ of the walk started at~$x$  and killed upon its first exit from~$V$.

We will focus on the asymptotic properties of $h^{D_N}$ in a sequence of lattice domains  $\{D_N\}_{N\ge1}$ arising, more or less, through scaling by~$N$, from a nice continuum domain $D\subseteq\R^2$.
 Here the  existence of 
 a limit  law  for  $\max_{x\in D_N}h^{D_N}_x-m_N$,  with the centering sequence  given by
\begin{equation}
	m_N:=2\sqrt g\log N-\frac34\sqrt g\log\log(N\vee\texte)
\end{equation}
for a suitable constant~$g$ (see \eqref{E:2.3a}), was settled in Bramson and Zeitouni~\cite{BZ} and Bramson, Ding and Zei\-touni~\cite{BDingZ}. 
In~\cite{BL1,BL2,BL3}, two of the present authors  strengthened  this to a full description of the \emph{extremal} process. This is expressed via the point process convergence
\begin{equation}
	\label{E:1.3}
	\sum_{x\in D_N}\delta_{x/N}\otimes\delta_{h_x^{D_N}-m_N}\,\,\,\underset{N\to\infty}\Lawarrow\,\,\,\sum_{i\ge1}\sum_{z\in\Z^2}\delta_{x_i}\otimes\delta_{h_i-\phi_z^{(i)}},
\end{equation}
where the process on the  right $\{(x_i,h_i,\phi^{(i)})\colon i\ge1\}$ enumerates the sample ``points'' of the Poisson point process
\begin{equation}
	\label{E:1.3a}
	\text{PPP}\Bigl(\bar c\,Z^D(\textd x)\otimes\texte^{-\alpha h}\textd h\otimes\nu(\textd\phi)\Bigr)
\end{equation}
on $D\times\R\times[0,\infty)^{\Z^2}$  and  the various objects  in \eqref{E:1.3a} are  as follows: $\bar c\in(0,\infty)$ is a constant,~$Z^D$ is a (non-degenerate) \emph{random} a.s.-finite Borel measure on~$D$ normalized so that
\begin{equation}
	\label{E:1.4}
	\lim_{s\downarrow0}\frac{E\bigl(Z^D(A)\,\texte^{- sZ^D(D)}\bigr)}{\log(1/s)} = \int_A r^D(x)^2\textd x,
\end{equation}
where~$r^D(x)$ is, for~$D$ simply connected, the conformal radius of~$D$ from~$x$ (see \eqref{E:3.3}), $\alpha:=2/\sqrt g$ and~$\nu$ is a (deterministic) probability measure on~$[0,\infty)^{\Z^2}$.

The random  measure~$Z^D$ encodes the macroscopic correlations of the DGFF that ``survive'' the scaling limit. As was shown in~\cite[Theorem~2.9]{BL2}, $Z^D$ can be independently characterized as a version of the \emph{critical} Liouville Quantum Gravity (cLQG).
The LQG is a one-parameter family of random measures introduced by Duplantier and Sheffield~\cite{DS} (although the concept goes back to Kahane's~\cite{Kahane} theory of multiplicative chaos) given formally as $\Ind_D(x)\texte^{\beta h(x)}\textd x$ for~$h$ denoting the continuum Gaussian Free Field (CGFF) on~$D$. However, since the CGFF exists only in the sense of distributions, work is needed to give this expression a rigorous meaning. This is particularly subtle in the critical case $\beta=\alpha$, which is the one relevant for the present paper.

 A natural way to construct the critical LQG is by regularization. This is the approach of Duplantier, Rhodes, Sheffield and Vargas~\cite{DRSV1,DRSV2}. (Another approach via ``mating of trees'' has recently been developed  in Duplantier, Sheffield and Miller~\cite{DSM21} and Aru, Holden, Powell and Sun~\cite{AHPS}). A number of different regularizations have been considered which then raises the question of uniqueness; i.e., independence of the regularization used.  This was first addressed  in Rhodes and Vargas~\cite{RV-review} and ultimately settled in Junnila and Saksman~\cite{JS} and Powell~\cite{Powell}.

In another paper by two of the present authors~\cite{BL4}, the subcritical ($0<\beta<\alpha$) LQGs have been shown to describe the spatial part of the $N\to\infty$ limit of random measures
\begin{equation}
	\frac1{K_N}\sum_{x\in D_N}\delta_{x/N}\otimes\delta_{h_x^{D_N}-2\sqrt g\lambda\log N},
\end{equation}
where $\lambda:=\beta/\alpha\in(0,1)$ and $K_N:=N^{2(1-\lambda^2)}/\sqrt{\log N}$. This  made it possible  to quantify the  growth-rate  of  the level set $\{x\in D_N\colon h_x^{D_N}\ge 2\sqrt g\lambda\log N\}$, called ``intermediate'' in~\cite{BL4} but also known as the ``DGFF thick points.'' 

The goal of the present paper is to derive similar control for (what we call) the \emph{near-extremal} level sets
\begin{equation}
	\label{e:1.2}
	\bigl\{x\in D_N\colon h^{D_N}_x\ge m_N- r\sqrt{\log N}\bigr\}
\end{equation}
for~$r>0$.
We will again glean this from a corresponding near-extremal point process,
\begin{equation}
	\label{e:1.3}
	\zeta_N^D:=\frac1{\log N}\sum_{x\in D_N}\texte^{\alpha(h^{D_N}_x-m_N)}\delta_{x/N}\otimes\delta_{(m_N-h^{D_N}_x)/{\sqrt{\log N}}}\,.
\end{equation}
It is actually already known that, modulo normalization, the spatial part of this measure, 
\begin{equation}
	\label{e:1.5}
	\frac1{\log N}\sum_{x\in D_N}\texte^{\alpha(h^{D_N}_x-m_N)}\delta_{x/N}
\end{equation}
to be referred to as the \emph{lattice-cLQG measure} below, converges in law to  a multiple of  $Z^D$ as~$N\to\infty$ (Rhodes and Vargas~\cite[Theorem 5.13]{RV-review} and Junnila and Saksman~\cite[Theorem 1.1]{JS}).  The measure \eqref{e:1.3}  extends \eqref{e:1.5} to include  information about  the statistics of the field values in the near-extremal regime.

Our main findings are as follows: First, the bulk of the lattice-cLQG measure is carried by  the near-extremal values; i.e., the  sites  where $m_N-h_x=O(\sqrt{\log N})$ with the scaled difference asymptotically Rayleigh distributed. Second, the convergence of the near-extremal process holds \emph{jointly} with that of the extremal process in \eqref{E:1.3} \textit{and}  the underlying DGFF.
Third, the $Z^D$-measure arising in the limit of the near-extremal process coincides, in a path-wise sense, with the intensity in \eqref{E:1.3a}  \textit{and} with  the cLQG defined by the limiting CGFF.
 Fourth,  this shows that,  while the limit near-extremal process is a measurable function of the CGFF, the extremal process is~not.   Fifth, we  also explain why the various ways to ``norm'' the measure \eqref{e:1.5} lead to the same limit continuum object, up to a deterministic multiplicative~constant.

\section{Main result}
\label{sec2}\noindent
The statement of the main theorem requires additional notation and definitions. These will be consistent with those used in references~\cite{BL1,BL2,BL3} and the review~\cite{B-notes},  where we draw useful facts from on several occasions in this paper.

For~$A\subseteq\Z^2$, let $\partial A$ denote the set of vertices in~$A^\cc$ that have a neighbor in~$A$.  We will use $\lfloor x\rfloor$ to denote the unique $z\in\Z^2$ such that $x- z\in[0,1)^2$. We write~$\textd_\infty$ for the~$\ell^\infty$-distance on~$\R^2$. The standard notation~$\NN(\mu,\sigma^2)$ is used for the law of a normal with mean~$\mu$ and variance~$\sigma^2$. If $\mu$ is a measure and~$f$ is a $\mu$-integrable function, we abbreviate $\langle\mu,f\rangle :=\int f\rmd\mu$. If~$G$ is a function of two variables, we write $G\mu$ for $x\mapsto \int G(x,y)\mu(\textd y)$. Given a topological space~$E$ we use $C_\cc(E)$ to denote the space of real-valued continuous functions on~$E$ with compact support in~$E$.

Let~$\mathfrak D$  denote  the class of bounded open subsets of~$\R^2$ whose boundary consists of a finite number of connected components each of which has a positive Euclidean diameter. We will approximate $D\in\mathfrak D$ by a sequence of  $\{D_N\}_{N\ge1}$ of lattice domains defined by
\begin{equation}
	\label{E:2.1}
	D_N := \bigl\{x \in \Z^2 \colon \textd_\infty(x/N,D^\cc)>\ffrac{1}{N}\bigr\}.
\end{equation}
Much of what is to follow depends on a precise control of the Green function in~$D_N$, or other domains of size of order~$N$.
Two specific constants,
\begin{equation}
	\label{E:2.3a}
	g:=\frac2\pi \quad\text{and}\quad c_0:=\frac{2\gamma+\log 8}\pi
\end{equation}
with~$\gamma$ denoting the Euler constant, appear throughout the derivations in the expansions of the diagonal Green function; see Lemma~\ref{lemma-3.2}.  These arise through an asymptotic formula of the potential kernel. 

 A key point driving several limit arguments is that, for  sites  in~$D_N$ separated by distances of order~$N$, the lattice Green function $G^{D_N}$ is well approximated by its continuum counterpart $\wh G^D\colon D\times D\to\R\cup\{+\infty\}$. One way to view~$\wh G^D$ is as the integral kernel for the inverse of the (scaled negative) Laplacian $-\frac14\Delta$ with Dirichlet boundary conditions on~$\partial D$. This means that $f(y):=\wh G^D(x,y)$ solves the Poisson equation $-\Delta f=4\delta_x$ in~$D$ with zero boundary conditions on~$\partial D$. As a result, we get an explicit representation using the Poisson kernel, cf~\eqref{E:3.6}, which shows that~$\wh G^D$ is continuous except on the diagonal of~$D\times D$ where it diverges logarithmically with~$|x-y|$. 

Let~$\mathcal{M}_\cc(D)$ be the set of finite signed  measures on~$\R^2$ of the form  $\rho = \rho_+ - \rho_-$, where~$\rho_\pm$ are non-negative finite measures with compact support in~$D$ and
\begin{equation}
	\int_{D \times D} \,\bigl|\wh G^D(x,y)\bigr|\, |\rho|(\rmd x) |\rho| (\rmd y)< \infty
\end{equation}
holds for~$|\rho|:=\rho_++\rho_-$. 
Following, e.g., Berestycki~\cite{Berestycki}, we then introduce:

\begin{definition}
\label{def-2.1}
	The Continuum Gaussian Free Field (CGFF) in~$D$ is a family of random variables $\{h^D(\rho)\colon \rho\in \mathcal{M}_\cc(D)\}$ such that $\rho\mapsto h^D(\rho)$ is linear and
	\begin{equation}
		h^D(\rho)=\NN(0,\sigma_\rho^2)\quad\text{\rm with}\quad \sigma_\rho^2:=\int_{D\times D} \wh  G^D(x,y) \rho(\rmd x) \rho(\rmd y) 
	\end{equation}
	for each $\rho\in \mathcal{M}_\cc(D)$.
\end{definition}

\noindent
Regarding~$\rho\mapsto h^D(\rho)$ as a coordinate  map,  we may and will henceforth treat~$h^D$ as an element of~$\R^{\mathcal{M}_\cc(D)}$ endowed with the product topology.

In order to define the critical LQG measure, we follow the approach in Duplantier and Sheffield~\cite{DS} that proceeds by smoothing~$h^D$ via circle averages. We do this mainly to ensure that the resulting critical LQG measure is measurable with respect to $h^D$. For each $r > 0$ and $x \in \bbR^2$,   let $\rho_{x,r}$ be the uniform measure on the Euclidean circle of radius~$r$ centered at~$x$.  Given $t > 0$ we then
set
\begin{equation}\label{e:ht-circ}
	h^D_t(x) := h^D(\rho_{x, \rme^{-t}}) \,.
\end{equation}
By~\cite[Proposition 3.1]{DS}, the family of random variables $\{h_t^D(x)\colon x \in D, t > 0\}$ admits a  version that  is jointly continuous on~$D \times (0,\infty)$. We may therefore use it to define the (random) measure
\begin{equation}
	\label{E:2.7}
	Z_t^D(\textd x):=\Ind_D(x)\alpha\sqrt{t}\,\texte^{\alpha h^D_t(x)-\frac12\alpha^2\Var(h^D_t(x))}\,r^D(x)^2\textd x,
\end{equation}
where~$\alpha:=2/\sqrt{g}=\sqrt{2\pi}$ and $r^D(x)$ is (for~$D$ simply connected) the conformal radius of~$D$ from~$x$; see \eqref{E:3.3} for an explicit formula.

Thanks to \cite[Corollary~5.8 and Remark~5.9]{JS} and \cite[Theorem~10 and Definition~11]{DRSV2} (see also \cite[Theorem~2.8]{Powell}), there exists a random measure~$Z^D$ on~$D$ such that $Z_t^D(\rmd x) \to Z^D(\rmd x)$ weakly in probability as $t\to\infty$. Moreover,  \cite[Theorem~2.9]{BL2} shows that~$Z^D$ satisfies~\eqref{E:1.4}. In order to step away from the specifics, we put forward:


\begin{definition}
	\label{d:1}
	The critical Liouville Quantum Gravity (cLQG) on~$D$ associated with~$h^D$ is any version of~$Z^D$ that is measurable with respect to~$h^D$.
\end{definition}

\noindent
Let~$\overline D$  stand for  the closure of~$D$ and  let us use  $\overline\R:=\R\cup\{+\infty,-\infty\}$  to denote  the two-point compactification of $\R$. Our main result~is~then:

\begin{theorem}[Near-extremal process convergence]
	\label{thm-A}
	Let~$D\in\mathfrak D$ and let~$\{D_N\}_{N\ge1}$ be  defined by  \eqref{E:2.1}. Given a sample~$h^{D_N}$ of the DGFF in~$D_N$, define $\zeta_N^D$ by~\eqref{e:1.3} with $\alpha:=2/\!\sqrt{g}$ and regard~$h^{D_N}$ as an element of~$\R^{\mathcal{M}_\cc(D)}$ via
	\begin{equation}
		h^{D_N}(\rho):=\int_D  \rho(\rmd x) \,h^{D_N}_{\lfloor xN\rfloor},\quad \rho\in  \mathcal{M}_\cc(D). 
	\end{equation}
	Let~$\eta^D_N$ be the measure on the left of \eqref{E:1.3}.
	Then, relative to the vague topology on the space of Radon measures on~$\overline D\times\overline\R$, resp., $\overline D\times (\R\cup\{+\infty\})$
	for~$\zeta^D_N$, resp.,~$\eta^D_N$ and the product topology on~$\R^{\mathcal{M}_\cc(D)}$ for~$h^{D_N}$,
	\begin{equation}
		\label{E:2.10}
		\bigl(\zeta^D_N,\eta^D_N,h^{D_N}\bigr)\,\,\underset{N \to \infty}{\Lawarrow}\,\, (\zeta^D,\eta^D,h^D) \,,
	\end{equation}
 where the objects on the right-hand side are as follows:~$h^D$ is a CGFF in~$D$ (according to Definition~\ref{def-2.1}) and, for a version~$Z^D$ of cLQG associated with~$h^D$  as in Definition~\ref{d:1}, $\zeta^D$ is the measure 
	\begin{equation}
		\label{E:2.3}
		\zeta^D (\textd x\,\textd t) :=   c_\star\,Z^D(\textd x) \otimes \Ind_{[0,\infty)}(t)\,t\rme^{-t^2/(2g)}\,\textd t  \end{equation}
  with  $c_\star:=\frac1{2\sqrt{g}}\,\texte^{2c_0/g}$ (where~$c_0$ and~$g$ are as in~\eqref{E:2.3a})  while  $\eta^D$ is  the  point process  on the right of \eqref{E:1.3}  whose points are  drawn according to~\eqref{E:1.3a},  conditional on the sample of~$Z^D$ as above. 
\end{theorem}

In order to elucidate the structure of the joint law of the limiting triplet in \eqref{E:2.10},  note that the proof of Theorem~\ref{thm-A} actually shows 
\begin{equation}
\begin{aligned}
	E\bigl(&\texte^{-\langle\zeta^D,f_1\rangle-\langle\eta^D, f_2\rangle+ h^D(\rho)}\bigr)
	\\
	&=\texte^{\frac12\langle \rho , \wh G^D \rho\rangle}\,E\biggl(\exp\Bigl\{-c_\star\int_D Z^D(\textd x)\texte^{\alpha (\wh G^D\rho)(x)}\int_0^\infty\textd t\, t\texte^{-\frac{t^2}{2g}}\,f_1(x,t)\qquad\qquad\quad
	\\
	&\qquad\qquad-\bar c\int_D Z^D(\textd x)\texte^{\alpha (\wh G^D\rho)(x)}\int_{\R}\textd h\,\texte^{-\alpha h}\int_{\R^{\Z^2}}\nu(\textd\phi)\bigl(1-\texte^{- \sum_{z\in\Z^2}f_2(x,h-\phi_z)}\bigr)\Bigr\}\biggr),
\end{aligned}
\end{equation}
 for all continuous~$f_1,f_2\colon D\times\R\to[0,\infty)$ with compact support and all~$ \rho\in \mathcal{M}_\cc(D)$. Here the expectation on the right is with respect to~$Z^D$ and~$\wh G^D\rho$ is a function on~$D$ defined by $(\wh G^D\rho)(x):=\int_D \wh G^D(x,y)\rho(\textd y)$.

Observe that the limiting near-extremal process~$\zeta^D$ is a measurable function of~$Z^D$ in~\eqref{E:2.3}, which in turn is a measurable function of the field~$h^D$. However, the extremal process~$\eta^D$ contains an additional degree of randomness due to Poisson sampling in~\eqref{E:1.3a} and is thus \emph{not} measurable with respect to~$h^D$.

\begin{remark}
	The above conclusions assume a very specific (albeit canonical) discretization \eqref{E:2.1} of the continuum domain~$D\in\mathfrak D$.
	This is  stricter than (but included in)  the discretizations considered in \cite{BL1,BL2,BL3,BL4} where only
	\begin{equation}
		\label{E:2.1a}
		D_N\subseteq \bigl\{x\in\Z^2\colon \textd_\infty(\ffrac xN,D^\cc)>\ffrac1N\bigr\}
	\end{equation}
	and, for each~$\delta>0$,
	\begin{equation}
		\label{E:2.2a}
		\bigl\{x\in\Z^2\colon \textd_\infty(\ffrac xN,D^\cc)>\delta\bigr\}\subseteq D_N
	\end{equation}
	were assumed. As it turns out, the limit in Theorem~\ref{thm-A} remains in effect  even under the above weaker form of discretization, 
	provided that the convergence is claimed only for test functions compactly supported in~$D$, not~$\overline D$. The restriction on the support of the test function stems from our inability to control the tightness of $\zeta_N^D$ when~$D_N$ has  a large number of small ``holes'' near the boundary. 
\end{remark}

Since the second component of the measure in \eqref{E:2.3} is deterministic, Theorem~\ref{thm-A} implies the following limit result:

\begin{corollary}
	\label{cor-2.5}
	For all $\kappa\in\R$ with $0<\kappa<\tfrac1{2g}$  and all continuous functions $\Phi\colon\R\to\R$ for which
	$t\mapsto \Ind_{[0,\infty)}(t)\Phi(t)\rme^{-\kappa t^2}$
	is Lebesgue integrable, we have  
	\begin{equation}
		\label{E:2.14}
		\frac1{\log N}\sum_{x\in D_N}
		\Phi\Bigl(\frac{m_N-h^{D_N}_x}{\sqrt{\log N}}\Bigr)\,
		\texte^{\alpha(h^{D_N}_x-m_N)}\delta_{x/N}
		\,\,\,\underset{N \to \infty}{\Lawarrow}\,\,\,c(\Phi)\,Z^D 
	\end{equation}
	relative to the vague topology on  Radon measures on  $D$, where $c(\Phi):=c_\star\int_0^\infty\Phi(t)t\texte^{-t^2/(2g)}\textd t$.
\end{corollary}

\noindent
This  corollary  explains why the various known ways to ``norm'' the lattice-cLQG measure lead to the same limiting object, except for an overall (deterministic) normalization constant. Besides $\Phi(t)=1$, which reduces to \eqref{e:1.5}, the choice $\Phi(t):=t$ leads to the so called derivative martingale.

The near-extremal process will likely find its natural counterparts in the studies of other logarithmically-correlated fields. This includes the problem of the most visited site for random walks in planar domains; see~\cite{AB,ABL} and~\cite{Jego} for recent  results on the ``thick points'' in that context. Our conclusions are also  used in a companion paper~\cite{BGL} to control fractal properties of the~cLQG.

The remainder of the present paper is organized as follows. In Section~\ref{sec3} we recall a number of known facts about the DGFF, its extreme values and connection to the ballot theorems. Section~\ref{sec-4a}  establishes useful conditions for identifying random measures with the cLQG. Section~\ref{sec4} is devoted to calculations and estimates of moments of the measure~$\zeta_N^D$ which are then used, in Section~\ref{sec5}, to prove Theorem~\ref{thm-A}  and Corollary~\ref{cor-2.5}. 

\section{Preliminaries}
\label{sec3}
\noindent 
The exposition of the proofs starts by collecting various preliminary facts about the Green functions, DGFF and its extreme values and the ballot estimates for the DGFF. An uninterested (or otherwise informed) reader may consider skipping this section and returning to it only when the stated facts are used in later proofs. 

\subsection{Green function}
 We start by discussing the relevant aspects of the Green function which, in light of it being the covariance of the DGFF, plays an important role throughout this work. First we need  some notation. 
Some of the estimates in what follows will hold only in domains slightly smaller than~$D$, and so for each $\delta \geq 0$, we set
\begin{equation}
	D^\delta:=\{x\in D\colon \rmd_\infty(x,D^\cc)>\delta\} \,.
\end{equation}
We will write $D_N^\delta$ for the discretization of $D^\delta$ as specified in~\eqref{E:2.1}.

 Next, in order to  address the large-$N$ behavior of the DGFF, we need
the large-$N$ behavior of the Green function~$G^{D_N}$ in lattice domain~$D_N$ approximating a continuum domain~$D\in\fU$ via \eqref{E:2.1}. For this, let~$\Pi^D(x,\cdot)$ be the harmonic measure on~$\partial D$, defined, e.g., as the exit distribution from~$D$ of a Brownian motion started at $x$. 
Writing $\|\cdot\|_2$  for  the Euclidean norm on~$\R^2$, for each $x\in D$ we set
\begin{equation}
	\label{E:3.3}
	r^D(x):=\exp\left\{\int_{\partial D} \Pi^D(x,\rmd z)\log\|x-z\|_2\right\} 
\end{equation}
 and let~$r^D(x):=0$ for~$x\not\in D$.
 As is readily checked,  for~$D$ simply connected, this coincides with the  notion of  conformal radius of~$D$ from~$x$.
The continuum Green function can explicitly be  given by 
\begin{equation}
	\label{E:3.6}
	\wh G^D(x,y) :=-g\log\|x-y\|_2+g\int_{\partial D} \Pi^D(x,\rmd z)\log\|y-z\|_2,
\end{equation}
where~$g$ is as in \eqref{E:2.3a}. We then have:

\begin{lemma}[Green function asymptotic]
	\label{lemma-3.2}
	For all~$D\in\mathfrak D$, all sequences of domains $\{D_N\}_{N\ge1}$ obeying \twoeqref{E:2.1a}{E:2.2a} and  all $\delta>0$,
	\begin{equation}
		\label{eq:Green}
		\lim_{N\to\infty}\,\sup_{x\in D^\delta}\,\Bigl|\, G^{D_N}\bigl(\lfloor xN\rfloor,\lfloor x N\rfloor\bigr)-g\log N-c_0- g\log r^D(x)\Bigr|=0\,,
	\end{equation}
	where $g$ and $c_0$ are the constants from~\eqref{E:2.3a}. In addition, for all $\delta>0$,
	\begin{equation}
		\lim_{N\to\infty}\,\sup_{\begin{subarray}{c}
				x,y\in D^\delta\\|x-y|>\delta
		\end{subarray}}
		\Bigl|\,G^{D_N}\bigl(\lfloor xN\rfloor,\lfloor y N\rfloor\bigr)-\wh G^D(x,y)\Bigr|=0.
	\end{equation}
\end{lemma}

\begin{proof}
	This follows  from, e.g.,  Theorem~4.4.4 and Proposition~4.6.2
	of~\cite{Lawler-Limic} in conjunction with~\cite[Lemma~A.1]{BL2}. See \cite[Theorem~1.17]{B-notes} for a full proof.
\end{proof}

We will also need  the following bounds: 

\begin{lemma}
	\label{lemma-3.2a}
	For each~$D\in\mathfrak D$ there exists~$c=c(D)>0$ such that,  with  domains $\{D_N\}_{N\ge1}$ as in \eqref{E:2.1},
	\begin{equation}
		\label{E:3.6a}
		g\log \rmd_\infty(x, (D_N)^\cc)  -c\le G^{D_N}(x,x)\le g\log \rmd_\infty(x,(D_N)^\cc)+c
	\end{equation}
 holds for all~$N\ge1$ and all~$x\in D_N$. 
\end{lemma}

\begin{proof}
 The proof uses the properties of two-dimensional simple random walk along with the particular form we discretize~$D$. Let~$\fraka$ denote the potential kernel of simple random walk  on~$\Z^2$. Using the constants in \eqref{E:2.3a}, the asymptotic~form 
	\begin{equation}
		\label{E:3.7}
		\fraka(x)=g\log\Vert x\Vert_2+c_0+O(\Vert x\Vert_2^{-2}),\quad\Vert x\Vert_2\to\infty,
	\end{equation}
	 holds  by, e.g.,~\cite[Theorem 4.4.4]{Lawler-Limic}.
	 Next, let us write~$H^U(x,z)$  for the probability that simple random walk started at~$x\in U$ exits~$U$ at~$z$. The diagonal Green function admits the representation
	\begin{equation}
		\label{E:3.6b}
		G^{D_N}(x,x)=\sum_{z\in\partial D_N}H^{D_N}(x,z)\fraka(x-z),
	\end{equation}
	by, e.g.,~\cite[Theorem~4.6.2]{Lawler-Limic}. 
	 Denoting  $A_k:=\{z\in\Z^2\colon \texte^k\le \rmd_\infty(x,z)<  \texte^{k+1}\}$  and using the fact that norms on~$\R^2$ are comparable we then get
	\begin{equation}
		\label{E:3.8a}
		\Bigl|G^{D_N}(x,x)- g\log \rmd_\infty(x,(D_N)^\cc) \Bigr|
		\le \sum_{k=k_0}^{ k_1}
		H^{D_N}(x, A_k\cap \partial D_N) g(k-k_0) + c',
	\end{equation}
	where  $k_0 := \lfloor \log \rmd_\infty(x,(D_N)^\cc)\rfloor$ and  $k_1:=\lceil\log (N\diam D)\rceil$, and where  $c'<\infty$ is a constant.  It remains to estimate the sum uniformly in all parameters.
	
	\newcounter{obrazek}
	
	\begin{figure}[t]
		\refstepcounter{obrazek}
		\centerline{\includegraphics[width=3truein]{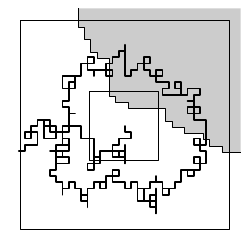}}
		\caption{An illustration of event~$E_j$ in the proof of Lemma~\ref{lemma-3.2a}. A path of  the  simple random walk started at the center point winds around the annulus~$A_j$ before exiting through the outer perimeter of~$A_j$. The path necessarily crosses the shaded area which represents the complement of~$D_N$.}
		\label{fig1}
	\end{figure}

	 For $j\ge0$, let $E_j$  denote the event that, for simple random walk started at~$x$,  the trajectory separates the inner and the outer boundaries of~$A_j$ before first hitting the outer boundary of $A_j$. 
	By coupling to Brownian motion and using  its  scale invariance,  there exists  $\tilde c>0$ such that~$E_j$ occurs with probability at least $1-\rme^{-\tilde c}$ for all~$j\ge 0$.  Next observe that,  by our assumptions, there exists $\ep>0$ such that each connected component of~$\partial D$ has diameter at least $\ep$.  
	Along with our way to discretize~$D$ into~$D_N$ from~\eqref{E:2.1} this ensures that, for  all~$j$ with $k_0< j<\lfloor\log(\ep N/4)\rfloor$,  the random walk trajectory  must intersect  $\partial D_N\cap A_j$ whenever~$E_j$ occurs.  See Fig.~\ref{fig1} for an illustration. 
	
 Using the strong Markov property successively at the times the walk hits the outer  boundary of~$A_j$, we thus obtain the geometric bound
	\begin{equation}
		\label{E:3.8}
		H^{D_N}(x, A_k\cap \partial D_N)\le\texte^{-\tilde c(\min\{k,k'\}-k_0  -1)} \,,
	\end{equation}
	 for all~$k$ with  $k_0< k<\lfloor\log(\ep N/2)\rfloor$, where $k':=\lfloor\log(\ep N/4)\rfloor$. As $k_1-k'$  is bounded uniformly in $N$  by a constant  depending only on $\ep$,  we can replace $\min\{k,k'\}$ by~$k$ in~\eqref{E:3.8}  at the cost of a multiplicative constant popping up in front of the exponential. 
	Using this the sum~\eqref{E:3.8a} is bounded uniformly as desired.
\end{proof}

We remark that Lemma~\ref{lemma-3.2a} is the reason why we have to work with the discretization \eqref{E:2.1} instead of \twoeqref{E:2.1a}{E:2.2a}, for which the upper bound \eqref{E:3.6a} fails in general.

\subsection{Gibbs-Markov property}
One of the important properties of the DGFF is the behavior under the restriction to a subdomain. As this arises directly from the Gibbsian structure of the law of the DGFF and is similar to the Markov property for stochastic processes, we refer to this as the Gibbs-Markov property, although the term domain-Markov has been used as well.

\begin{lemma}[Gibbs-Markov property]
	\label{lem:Gibbs-Markov}
	Let $\emptyset\ne U\subseteq V\subsetneq\Z^2$ and let~$h^V$ be the DGFF in~$V$. Define
	\begin{equation}
		\varphi^{V,U}(x) := \E\bigl(h^V_x\,\big|\,\sigma(h^V_y\colon y\in V\smallsetminus U)\bigr).
	\end{equation}
	Then
	\begin{equation}
		h^V-\varphi^{V,U}\,\independent\,\varphi^{V,U}\quad\text{and}\quad
		h^V-\varphi^{V,U}\laweq h^U.
	\end{equation}
	Moreover, a.e.\ sample of~$\varphi^{V,U}$ is discrete harmonic on~$U$.
\end{lemma}

\begin{proof}
	See, e.\,g.,~\cite[Lemma~3.1]{B-notes}.
\end{proof}

Consider now two sequences $\{D_N\}_{N\ge1}$, resp., $\{\wt D_N\}_{N\ge1}$ of lattice domains related, via \eqref{E:2.1}, to continuum domains $D$, resp.,~$\wt D$ subject to $\wt D\subseteq D$. While the DGFFs $h^{D_N}$, resp., $h^{\wt D_N}$ do not allow for pointwise limits, the  ``binding''  field $\varphi^{D_N,\wt D_N}$ does, due to controlled behavior of the variances:

\begin{lemma}
	\label{lemma-3.3}
	Let~$D,\wt D\in\mathfrak D$ obey $\wt D\subseteq D$. Define
	\begin{equation}
		\label{e:2.12}
		C^{D,\wt D}(x,y):=g\int_{\partial D}\Pi^D(x,\rmd z)\log\|y-z\|_2
		-g\int_{\partial \wt D}\Pi^{\wt D}(x,\rmd z)\log\|y-z\|_2.
	\end{equation}
	Then $x,y\mapsto C^{D,\wt D}(x,y)$, for $x,y\in\wt D$, is symmetric, positive semi-definite, continuous and harmonic in each variable.
 Moreover, there exists a centered Gaussian process~$\Phi^{D,\wt D}=\NN(0,C^{D,\wt D})$ with continuous and harmonic sample paths.
\end{lemma}

\begin{proof}
	See~\cite[Lemma 2.3]{BL2}.
\end{proof}

The equicontinuity of the sample paths of~$\Phi^{D,\wt D}$ implied by harmonicity permits a locally-uniform coupling of the discrete and continuum processes:

\begin{lemma}
	\label{lemma-3.5}
	Assume the setting of Lemma~\ref{lemma-3.3} and let $\{D_N\}_{N\ge1}$, resp., $\{\wt D_N\}_{N\ge1}$ be lattice domains approximating~$D$, resp.,~$\wt D$ in the sense \twoeqref{E:2.1a}{E:2.2a}. Suppose $\wt D_N\subseteq D_N$ for all~$N\ge1$. Then for each~$N\ge1$ there is a coupling of $\varphi^{D_N,\wt D_N}$ and $\Phi^{D,\wt D}$ such that, for each $\delta > 0$,
	\begin{equation}
		\label{eq:bind-coupl}
		\sup_{x\in \wt D^\delta}\,\Bigl|\varphi^{D_N,\wt D_N}(\lfloor xN\rfloor)-\Phi^{D,\wt D}(x)\Bigr|\longrightarrow 0
	\end{equation}
	in probability as $N\to\infty$.
\end{lemma}
 
\begin{proof}
	See \cite[Appendix]{BL2} or \cite[Lemma~4.4]{B-notes}.
\end{proof}

Another fact we will need is that the ``binding'' field $\varphi^{D_N,A}$ between  lattice domains  $A$ and~$D_N$ can be bounded uniformly for~$A$ with~$D_N^\delta\subseteq A\subseteq D_N$:

\begin{lemma}
	\label{l:cd}
	Let $D\in\mathfrak D$. For each compact $B\subseteq D$ and all $t>0$, we have
	\begin{equation}
		\lim_{\delta \downarrow 0} \limsup_{N \to \infty} \sup_{D^\delta_N\subseteq A\subseteq D_N} \bbP\Bigl(\,\max_{\substack{x\in D_N\\ x/N\in B}}\big|\varphi^{D_N,A}_x\bigr|>t\Bigr) = 0 \,.
	\end{equation}
\end{lemma}

\begin{proof}
 This is a discrete analogue of~\cite[Proposition~3.7]{BL2}. For a formal proof,  use the derivation of~(B.17) in~\cite[Lemma~B.11]{Ballot} with the following  choices: $n:=\log N$, $U:=D$ and~$A:=\wt U_n$.
\end{proof}

\subsection{Ballot estimates}
\label{sec-ballot}\noindent
Our proofs of the  near-extremal process convergence rely heavily on computations involving the tail asymptotic for the DGFF maximum. A starting point is the precise bound (to within  a multiplicative  constant) on the upper tail:

\begin{lemma}
	\label{lemma-DGFF-tail}
 Let $D\in\mathfrak D$. There exists a constant $c<\infty$ such that for all $N\ge1$ and all $u\in[1,\sqrt{\log N})$,
	\begin{equation}
		\label{e:A1}
		\bbP\bigl(\,\max_{x\in D_N} h^{D_N}_x> m_N+u\bigr)\leq c u\rme^{-\alpha u} \,.
	\end{equation}
\end{lemma}

\begin{proof}
	For~$D_N:=(0,N)^2\cap\Z^2$, this is~\cite[Theorem~1.4]{DZ}. A routine extension to general domains is performed in \cite[Proposition~3.3]{BL2}.
\end{proof}

In addition, we will need to control the tail probability even under conditioning on the field at a vertex to be fairly large. We begin with an asymptotic statement:

\begin{lemma}
	\label{prop:asymptotic}
	 Let $D\in\mathfrak D$.
	There exists a function $\cR\colon[0,\infty)\to(0,\infty)$ and, for each $x\in D$, a function $\cL^x \colon[0,\infty) \to(0,\infty)$ such that the following holds:
	\begin{enumerate}
		\item[(i)] For all $\delta > 0$,  uniformly in $x \in D^\delta$, 
		\begin{equation}
			\label{e:lem:rw-as:fg}
			\lim_{s\to\infty} \frac{\cL^x(s)}s = \lim_{s\to\infty}\frac{\cR(s)}s = 1.
		\end{equation}
		\item[(ii)] For all $x \in D$, $\delta>0$  and all  $u \geq 1$, the quantity~$o(1)$ defined by
		\begin{equation}
			\label{E:3.11}
			\bbP\bigl(h^{D_N}\leq m_N+u\,\big|\, h^{D_N}_{\lfloor Nx \rfloor}=m_N-s\bigr)
			=\bigl(2+o(1)\bigr)\,\dfrac{\cL^{x}(u)\cR(u+s)}{g\log N}
		\end{equation}
		tends to zero as $N\to\infty$ uniformly in~$s\in[0,(\log N)^{1-\delta}]$.
	\end{enumerate}
\end{lemma}

\begin{proof}
 For the proof, and also later use,  denote 
	\begin{equation}
		B(x,r):=\bigl\{y\in\R^d\colon \textd_\infty(x,y)<r\bigr\}.
	\end{equation}
	 The proof draws heavily on the conclusions of the technical paper~\cite{Ballot} by two of the present authors so we give only the main steps. Given $N\in\bbN$ and  $x \in D$,  denote  $D^{(N)}:=-\lfloor Nx\rfloor /N + D$.  Then $D^{(N)}\in\mathfrak D$ and  $D^{(N)}_N=-\lfloor Nx\rfloor+D_N$ and so the left-hand side of~\eqref{E:3.11} equals
	\begin{equation}
		\label{E:3.20ui}
		\bbP\Bigl(h^{D^{(N)}_N}-m_N-u\leq 0\,\Big|\, h^{D^{(N)}_N}_{0}-m_N-u=-s-u\Bigr).
	\end{equation}
	By~\cite[Theorem~1.1]{Ballot},  in which $\cL_{n,0, U}$ is a function on the outer boundary of a discretization~$U_n$ of the continuum domain~$U$ and $\cR_{0,0,V}$ is a function on a discretized inner boundary of the continuum domain~$V$, the probability in~\eqref{E:3.20ui}~equals 
	\begin{equation}
		\bigl(2+o(1)\bigr)\frac{\cL_{\log N,0, D^{(N)}}(-u \Ind_{\partial D^{(N)}_N})\cR_{0,0,B(0,1/4)}((-u-s) \Ind_{\{0\}} )}{g\log N},
	\end{equation}
 where~$o(1)\to0$ 	in the limit as $N \to \infty$, uniformly  in~$s\in[0,(\log N)^{1-\delta}]$  for  any  fixed $u \geq 1$. Here  the functions on the right-hand side are written in  the notation of~\cite{Ballot}  with the discretization of the continuum domain~$D^{(N)}$  for $\cL_{\log N,0, D^{(N)}}$ chosen as in the present paper (see~\cite[Remark~2]{Ballot}), while that for $\cR_{0,0,B(0,1/4)}$ as in~\cite[Eq.~(1.3)]{Ballot}.  (The inner boundary in the latter function is then just the origin; hence the appearance of~$1_{\{0\}}$ in the argument.) 
	
	By the continuity of the $\cL$-functional in the underlying domain  (see \cite[Proposition~1.6]{Ballot}), 
	$\cL_{\log N,0, D^{(N)}}(-u\Ind_{\partial D^{(N)}_N})$  has the same $N\to\infty$ asymptotic as its continuum analogue   $\cL_{\log N,0,-x+D}( -u \Ind_{\partial (-x+D)_N})$  which by~\cite[Proposition~1.3]{Ballot} converges, as~$N\to\infty$,  to some $ \cL^x(u)$.
	Setting  $\cR(s):=\cR_{0,0,B(0,1/4)}(-s \Ind_{\{0\}})$, we  arrive at~\eqref{E:3.11}.  Proposition~1.4 of~\cite{Ballot}  then  yields~\eqref{e:lem:rw-as:fg}.
\end{proof}

We remark that asymptotic expressions of the form~\eqref{E:3.11} lie at the heart of the work~\cite{BL3} where they were derived using the so called concentric decomposition. In~\cite{B-notes}, the concentric decomposition was used to give alternative proofs of tightness and convergence in law of  the centered DGFF maximum  that avoid the use of the modified branching random walk which  the earlier proofs by Bramson and Zeitouni~\cite{BZ} and Bramson, Ding and Zeitouni~\cite{BDingZ} relied on. (A key additional input on top of~\cite{BL3} was the first limit in \eqref{e:lem:rw-as:fg}; see~\cite[Proposition~12.5]{B-notes}.) Notwithstanding, Lemma~\ref{prop:asymptotic} is stronger than any of these prior conclusions; e.g., compared with \cite[Proposition~5.2]{BL3}. Indeed, there $-s=u$  and~$s$ has to be confined to  a finite interval. 

In~\cite{Ballot} the concentric decomposition  has been developed further  to yield uniform asymptotics as well as bounds for  various  ``ballot'' probabilities in more general setups.  This includes the  case where the field is conditioned on a subset of its domain which is not necessarily a single point,  as needed in the next lemma.  In the sequel,  we use the usual notation $a^+:=\max\{a,0\}$ and $a^-:=\max\{-a,0\}$.  For the purpose of the next lemma, we also extend the meaning of $h^{U}$ to DGFF with non-trivial boundary values. Namely, under $\bbP(\,\cdot\,\big|\,h^{U}|_{\partial U } = u)$, the field~$h^U$ is a multivariate normal indexed by vertices in~$U$ with covariance~$G^U$ and mean given by the unique harmonic extension of~$u$ onto~$U$.

\begin{lemma}
	\label{prop:UB}
	There is~$\delta_0\in(0,1)$ and, for each $\delta\in(0,\delta_0)$, there is $c=c(\delta)\in(0,\infty)$ such that the following holds for all $D, \wt{D} \in \mathfrak D$ with~$D$ connected, $B(0,\delta)  \subseteq  D\cap \wt D$,  $D\cup\wt D\subseteq B(0,1/\delta)$ and diameter of each connected component of $\partial D$ bounded from below by $\delta$: For all $1 \le K \leq  c^{-1} N$, all $u \in \bbR^{\partial D_N}$ and all $v \in \bbR^{\partial( \wt{D}^\cc)_K}$, abbreviating $U:=D_N \cap (\wt{D}^\cc)_K$,
	\begin{equation}
	\begin{aligned}
	\label{E:3.16iu}
		\bbP\Bigl(h^{U}|_{D^\delta_N \cap  (\wt{D}^c)^\delta_K}  \leq 0 \,\Big|\,h^U|_{\partial D_N} = -m_N+u \,,\,\,
		&h^U|_{\partial (\wt{D}^\cc)_K} = -m_K+v \Bigr) 
		\\
		&\leq c\, \frac{(1+u_\star^-)(\rme^{-(v_\star^+)^{3/2}\Ind_{\{u^-_\star \leq \sqrt{\log N/K}\}}} +v_\star^-)}{\log(N/K)}\,,
		\qquad
	\end{aligned}
	\end{equation}
	
	where
	\begin{equation}
		u_\star := \ol{u}_0 - \max_{x,y\in D_N^\delta} \big| \ol{u}_x - \ol{u}_y \big| \quad\text{and}\quad
		v_\star :=  \ol{v}_0-3\max_{x,y\in (\wt{D}^\cc)_K^\delta\cup\{0\}} \big|\ol{v}_x-\ol{v}_y\big|
	\end{equation}
	for $\ol{u}$, resp., $\ol{v}$ denoting the (unique) bounded harmonic extensions of $u$, resp., $v$ onto $D_N$, resp., $(\wt{D}^\cc)_K$.
	In addition, for all $u \geq 0$, $v \in \bbR$, $x \in D_N^\delta$, $\delta' \in[0,\delta^{-1}]$ and $N \geq 1$,
	\begin{equation}
		\label{E:3.17iu}
		\bbP\Bigl(h^{D_N}\leq m_N+u \text{\rm\ in }D^{\delta'}_N  \,\Big|\, h^{D_N}_x=m_N+v\Bigr)
		\leq c \frac{(1+u)(1+(u-v)^+)}{\log N}\,.
	\end{equation}
\end{lemma}

\begin{proof}
	This follows immediately from~\cite[Theorem~1.8]{Ballot} and \cite[Remark~2]{Ballot}. The constant~$3$ in the definition of $v_\star$ comes from the fact that in~\cite[Theorem~1.8]{Ballot}, $\ol{v}_\infty$ is used in place of $\ol{v}_0$.
\end{proof}

\section{Characterizations of cLQG measure}
\label{sec-4a}\noindent
In order to identify the measure arising in the limit \eqref{E:2.10} with the cLQG measure constructed by the limit of the measures in \eqref{E:2.7}, we will need conditions that characterize the cLQG measure uniquely. For subcritical LQG measures, such conditions have been found by Shamov~\cite{Shamov}, but the argument does not extend to the critical case. We will thus rely on \cite[Theorem~2.8]{BL2} in the following specific form:

\begin{theorem}
	\label{BL2-thm2.8}
	Suppose~$\{M^D\colon D\in\mathfrak D\}$ is a family of random Borel measures (not necessarily realized on the same probability space) that obey:
	\settowidth{\leftmargini}{(1111)}
	\begin{enumerate}
		\item[(1)] $M^D$ is concentrated on~$D$ and $M^D(D)<\infty$ a.s.
		\item[(2)] $P(M^D(A)>0)=0$ for any Borel $A\subseteq D$ with $\leb(A)=0$.
		\item[(3)] For any $a\in\R^2$, 
		\begin{equation}
			\label{E:2.17ie}
			M^{a+ D}(a+\textd x)\laweq  M^D(\textd x).
		\end{equation}
		\item[(4)] If $D,\widetilde D\in\mathfrak D$ are disjoint then
		\begin{equation}
			\label{E:2.18ua}
			M^{D\cup\wt D}(\textd x)\laweq M^D(\textd x)+M^{\wt D}(\textd x) \,,
		\end{equation}
		with $M^D$ and $M^{\wt D}$ on the right regarded as independent.
		If $D,\wt D\in\mathfrak D$ instead
		obey $\widetilde D\subseteq D$ and $\leb(D\smallsetminus\widetilde D)=0$, then (for $\alpha:=2/\sqrt g$)
		\begin{equation}
			\label{E:2.19ua}
			M^D(\textd x) \laweq \texte^{\alpha  \Phi^{ D, \wt{ D}}(x)} \,M^{\wt{ D}}(\textd x)\,,
		\end{equation}
		where $\Phi^{ D, \wt{ D}}$ is a centered Gaussian field with covariance $C^{D,\wt D}$, independent of $M^{\wt{ D}}$.
		\item[(5)] There is~$c\in(0,\infty)$ such that for  every $\delta\in(0,1)$ and all equilateral triangles $T,T_\delta\in\mathfrak{D}$ centered in $0$ with side-length $1$ and $1-\delta$, respectively, and the same orientation, we have
		\begin{equation}
			\label{E:1.25a}
			\lim_{\lambda\downarrow0}\,\,\sup_{K\ge1}\,\,\biggl|\,\frac{K^4\,E(M^{K^{-1}T}(K^{-1}T_\delta)\texte^{-\lambda K^4 M^{K^{-1}T}(K^{-1}T)})}{\log(\ffrac1\lambda)}- c\int_{T_\delta} r^{\,T}(x)^2\,\textd x\biggr|\,=\,0,
		\end{equation}
		where $r^T(x)$ is as in \eqref{E:3.3}.
	\end{enumerate}
	Then
	\begin{equation}
		\label{E:3.25i}
		M^D\laweq cZ^D,\quad D\in\mathfrak D,
	\end{equation}
	where $Z^D$ denotes $cLQG$ measure on $D$.
\end{theorem}

\begin{proof}
	 This is the statement of \cite[Theorem~2.8]{BL2} with the assumptions modified as detailed in~\cite[Remark~6.9]{BL2}.
\end{proof}

As it turns out, not only does Theorem~\ref{BL2-thm2.8}  determine   the law of cLQG measures uniquely but it also offers simple criteria that identify two such families of measures in path-wise sense. This is the content of:

\begin{theorem}
	\label{thm-3.10}
	Let $\{M^D\colon D\in\mathfrak D\}$, resp., $\{\wh M^D\colon D\in\mathfrak D\}$ be two families of random Borel measures satisfying conditions (1-3) and \eqref{E:2.18ua} in Theorem~\ref{BL2-thm2.8} and such that~\eqref{E:1.25a} holds with a constant~$c>0$ for $\{M^D\colon D\in\mathfrak D\}$ and a constant $\hat c>0$ for $\{\wh M^D\colon D\in\mathfrak D\}$. Assume that $M^D$ and~$\wh M^D$ are defined, for each~$D\in\mathfrak D$, on the same probability space so that \eqref{E:2.19ua} holds jointly for both measures, i.e., for all $D,\wt D\in\mathfrak D$ with $\wt D\subseteq D$ and~$\leb(D\smallsetminus\wt D)=0$, we have
	\begin{equation}
		\label{E:3.25u}
		\bigl(M^D(\textd x),\wh M^D(\textd x)\bigr) \laweq \Bigl(\texte^{\alpha  \Phi^{ D, \wt{ D}}(x)} \,M^{\wt{ D}}(\textd x),\texte^{\alpha  \Phi^{ D, \wt{ D}}(x)} \,\wh M^{\wt{ D}}(\textd x)\Bigr)\,,
	\end{equation}
	where $\Phi^{D,\wt D}$ is independent of~$(M^{\wt D},\wh M^{\wt D})$ on the right-hand side. Then
	\begin{equation}
		\label{E:3.26}
		c^{-1} M^D = \hat c^{-1}\wh M^D\quad \text{\rm a.s.}
	\end{equation}
	for each~$D\in\mathfrak D$.
\end{theorem}

\begin{proof}
	The proof could in principle be referred to  \cite[Theorem~6.1]{BL2} which states that any measure~$M^D$ satisfying conditions (1-5) can be represented, by sequential application of the Gibbs-Markov property \eqref{E:2.19ua}, as a weak limit of approximating measures that depend only on the ``binding'' fields $\Phi^{D,\wt D}$ used in \eqref{E:2.19ua}. We will instead provide an independent argument that only uses the conclusion of Theorem~\ref{BL2-thm2.8}.
	
	Given~$p\in(0,1)$, consider the random measure
	\begin{equation}
		\wt M_p^D:=pM^D+(1-p)\wh M^D.
	\end{equation}
	By our assumptions, the family $\{\wt M^D_p\colon D\in\mathfrak D\}$ satisfies conditions (1-4) of Theorem~\ref{BL2-thm2.8}. We claim that it satisfies also condition~(5) with constant
	\begin{equation}
		c_p:=p c+(1-p)\hat c.
	\end{equation} 
	For this pick $K\ge1$  and $\delta>0$, let $T, T_\delta$ be defined as in~(5),  abbreviate
	\begin{equation}
		X(\cdot):=K^4 M^{K^{-1} T  }(K^{-1}\cdot)\quad\text{and}\quad 
		Y(\cdot):=K^4\wh M^{K^{-1} T  }(K^{-1}\cdot)
	\end{equation}
	and consider the rewrite
	\begin{equation}
	\begin{aligned}
		\label{E:3.31ie}
		E\Bigl(K^4\wt M_p^{K^{-1} T  }(K^{-1} T_\delta  )&\texte^{-\lambda K^4\wt M_p^{K^{-1} T  }(K^{-1} T  )}\Bigr)
		\\
		&=
		 pE\bigl( X( T_\delta  )\texte^{-\lambda p X( T  )}\bigr)+(1-p)E\bigl(Y( T_\delta  )\texte^{-\lambda(1-p)Y( T  )}\bigr)
		\\
		&\qquad-pE\Bigl( X( T_\delta  )\texte^{-\lambda p X( T  )}\bigl(1-\texte^{-\lambda(1-p)Y( T  )}\bigr)\Bigr)		
		\\
		&\qquad\qquad-(1-p)E\Bigl( Y( T_\delta  )\texte^{-\lambda (1-p) Y( T  )}\bigl(1-\texte^{-\lambda p X( T  )}\bigr)\Bigr).
	\end{aligned}
	\end{equation}
	Invoking \eqref{E:1.25a} for the two families of measures, as~$\lambda\downarrow0$ (with~$p$ fixed), the sum of the first two terms on the right divided by~$\log(1/\lambda)$ tends to $c_p\int_{ T_\delta  } r^{ T } (x)^2\textd x$, uniformly in~$K\ge1$. We claim that the two remaining terms are $o(\log(1/\lambda))$. 
	
	Focusing only on the third term by symmetry, the Cauchy-Schwarz inequality bounds this by the square root of
	\begin{equation}
		\label{E:3.33}
		E\Bigl( X( T_\delta  )^2\texte^{-2\lambda p X( T  )}\Bigr)E\Bigl(\bigl(1-\texte^{-\lambda (1-p) Y( T  )}\bigr)^2\Bigr).
	\end{equation}
	Since Theorem~\ref{BL2-thm2.8} applied to~$X$ and~$Y$ separately ensures that~$X$ and~$Y$ have the (marginal) law of $K^4 Z^{K^{-1} T  }$ which has the law of~$Z^{ T }$ by exact scaling properties of~$Z^{ T  }$ under dilation (cf~\cite[Corollary~2.2]{BL2}), the quantities in \eqref{E:3.33} are independent of~$K\ge1$. Dominating $X( T_\delta  )^2\texte^{-2\lambda p X( T  )}\le\frac1{\lambda p}X( T  )\texte^{-\lambda p X( T  )}$ shows, with the help of \cite[Corollary~2.7]{BL2} that the first term in \eqref{E:3.33} is $O(\lambda^{-1}\log(1/\lambda))$. From the known properties of~$Z^{ T }$  (see again \cite[Corollary~2.7]{BL2}) we also get that the limit
	\begin{equation}
		\lim_{\lambda\downarrow0}\frac1{\lambda\log(\ffrac1\lambda)}\,E\bigl(1-\texte^{-\lambda Y( T  )}\bigr)
	\end{equation}
	exists and is finite. Then
	\begin{equation}
		E\Bigl(\bigl(1-\texte^{-\lambda Y( T  )}\bigr)^2\Bigr)=2E\bigl(1-\texte^{-\lambda Y( T  )}\bigr)-E\bigl(1-\texte^{-2\lambda Y( T  )}\bigr)
	\end{equation}
	shows that the second term in \eqref{E:3.33} is $o(\lambda\log(\ffrac1\lambda))$ as $\lambda\downarrow0$.  Hence, the quantity in \eqref{E:3.33} is $o(\log(\ffrac1\lambda)^2)$ and so the third (and by symmetry also the fourth) term on the right of \eqref{E:3.31ie} are $o(\log(\ffrac1\lambda))$, uniformly in~$K\ge1$.
	
	Having verified the conditions of Theorem~\ref{BL2-thm2.8}, 
	we thus get that for each~$p\in[0,1]$ and each~$D\in\mathfrak D$, 
	\begin{equation}
		\label{E:4.17w}
		pM^D+(1-p)\wh M^D \,\laweq \bigl(p c+(1-p)\hat c\bigr)Z^D
	\end{equation}
	Pick non-negative $f\in C_\cc(\overline D)$ and abbreviate $U:=c^{-1}\langle M^D,f\rangle$, $  V :=\hat c^{-1}\langle\wh M^D,f\rangle$ and $q:=pc/c_p$.  By \eqref{E:4.17w}  the law of the non-negative random variable $qU+(1-q) V $ does not depend on~$q\in[0,1]$. This implies $E\texte^{-(qU+(1-q)V)}=qE(\texte^{-U})+(1-q)E(\texte^{- V })$ and so, by Jensen's inequality, $\texte^{-(qU+(1-q) V  )}=q\texte^{-U}+(1-q)\texte^{- V }$ a.s. By Jensen's inequality again, this is only possible if $U=  V  $ a.s. As this holds for all non-negative~$f\in C_\cc(\overline D)$, we get \eqref{E:3.26} as desired.
\end{proof}

As a consequence of the above, we do get a characterization of cLQG by its behavior under the Cameron-Martin shifts of the underlying CGFF:

\begin{theorem}
	\label{thm-4.3}
	Let $\{M^D\colon D\in\mathfrak D\}$ be a family of random Radon measures satisfying the conditions of Theorem~\ref{BL2-thm2.8} with~$c:=1$ in \eqref{E:1.25a}. Suppose also that, for each~$D\in\mathfrak D$, there exists a coupling of~$M^D$ with the CGFF~$h^D$ such that for all~$\rho\in \mathcal{M}_\cc(D)$ and for~$\BbbP_\rho$ defined by
	\begin{equation}
		\BbbP_{\rho}(A):=\E\bigl(\texte^{h^D(\rho)-\frac12\Var(h^D(\rho))}\Ind_A\bigr),
	\end{equation}
	we have
	\begin{equation}
		\label{E:4.19iu}
		M^D\text{\rm\ under }\BbbP_{\rho}
		\,\,\,\laweq\,\,\, 
		\texte^{\alpha (\wh G^D\rho)(x)}M^D(\textd x)\text{\rm\ under }\BbbP.
	\end{equation}
	Then
	\begin{equation}
		\label{E:4.20io}
		M^D=Z^D\text{\rm\ \  a.s.}
	\end{equation}
	where $Z^D$ is the cLQG associated with~$h^D$.
\end{theorem}
\begin{proof}
	Let~$D,\wt D\in\mathfrak D$ be such that $\wt D\subseteq D$ and~$\leb(D\smallsetminus\wt D)=0$ and let~$\Phi^{D,\wt D}=\NN(0,C^{D,\wt D})$ be a sample of the ``binding'' field independent of~$M^{\wt D}$ and~$h^{\wt D}$. Let also~$\rho\in \mathcal{M}_\cc(\wt D)$ and $f\in C_\cc(\wt D)$ with $f\ge 0$. An elementary Gaussian integration and \eqref{E:4.19iu} then show 
	\begin{equation}
		\begin{aligned}
			\E\Bigl(\texte^{h^{\wt D}(\rho)+ \langle\rho,\Phi^{D,\wt D}\rangle }&\texte^{-\langle \texte^{\alpha\Phi^{D,\wt D}}M^{\wt D},f\rangle}\Bigr) 
			=\texte^{\frac12\langle \rho, C^{D,\wt D}\rho\rangle}
			\E\biggl(\exp\Bigl\{h^{\wt D}(\rho)-\bigl\langle \texte^{\alpha\Phi^{D,\wt D}+\alpha C^{D,\wt D}\rho}M^{\wt D},f\bigr\rangle\Bigr\}\biggr)
			\\
			&
			=\texte^{\frac12\langle \rho, G^{\wt D}\rho\rangle+\frac12\langle \rho, C^{D,\wt D}\rho\rangle}
			\E\biggl(\exp\Bigl\{-\bigl\langle \texte^{\alpha\Phi^{D,\wt D}+\alpha  C^{D,\wt D}\rho}\texte^{\alpha  G^{\wt D}\rho}M^{\wt D},f\bigr\rangle\Bigr\}\biggr).
		\end{aligned}
	\end{equation}
	The expectation on the right-hand side involves only~$M^{\wt D}$ and so is determined by the \emph{marginal} law of~$M^{\wt D}$ only. By assumption, this law obeys the Gibbs-Markov property meaning that $\texte^{\alpha\Phi^{D,\wt D}}M^{\wt D}\,\laweq\,M^D$. Using also that~$\wh G^{\wt D}+C^{D,\wt D}=\wh G^D$ and, one more time, \eqref{E:4.19iu} along with a Gaussian integral yields 
	\begin{equation}
		\E\Bigl(\texte^{h^{\wt D}(\rho)+\langle\rho,\Phi^{D,\wt D}\rangle}\texte^{-\langle \texte^{\alpha\Phi^{D,\wt D}}M^{\wt D},f\rangle}\Bigr) 
		=\E\Bigl(\texte^{h^{D}(\rho)}\texte^{-\langle M^{D},f\rangle}\Bigr).
	\end{equation}
	 Using the pair $\texte^{\alpha\Phi^{D,\wt D}} M^{\wt D},h^{\wt D}+\Phi^{D,\wt D}$ as a ``new'' definition of $(M^D,h^D)$ we may assume that $(M^D, h^D)$ and $( \texte^{\alpha\Phi^{D,\wt D}}  M^{\wt D}, h^{\wt D} + \Phi^{D,\wt D})$ are coupled so that
	\begin{equation}
		\label{e:shift-GM}
		\Bigl(\langle\texte^{\alpha\Phi^{D,\wt D}} M^{\wt D},f\rangle,h^{\wt D}(\rho)+\langle\Phi^{D,\wt D},\rho\rangle\Bigr)  =\bigl(\langle M^D,f\rangle,h^D(\rho)\bigr)
	\end{equation}
	holds pointwise for all~$f\in C_\cc(\wt D)$ and all~$\rho\in\MM_\cc(\wt D)$. 
	
	 Let $h^D_t$ and $\rho_{x,\rme^{-t}}$ be as in~\eqref{e:ht-circ} and assume containment in the (full measure) event that~$t,x\mapsto h_t(x)$ is jointly continuous.  By harmonicity of $\Phi^{D,\wt{D}}$, for any $x \in \wt D$ and~$t$  so large that~$\texte^{-t}$  becomes smaller than  the distance of~$x$ to the complement of~$\wt D$, 
	\begin{equation}
		\bigl\langle \rho_{x, \rme^{-t}},\Phi^{D,\wt{D}}\bigr\rangle= \Phi^{D,\wt{D}}(x) \,.
	\end{equation}
	 Under the coupling \eqref{e:shift-GM}  we then get
	\begin{equation}
		\label{E:4.23w}
		h^{\wt D}_t(x) + \Phi^{D,\wt{D}}(x)
		= \,\,\, h^D_t(x)
	\end{equation}
	 and, in particular, 
	\begin{equation}
		\Var\bigl(h^D_t(x)\bigr) - \Var\bigl(h^{\wt D}_t(x)\bigr) = \Var\bigl(\Phi^{D,\wt D}(x)\bigr)
	\end{equation}
	simultaneously for all such~$x$ and~$t$.
	
	The variance on the right-hand side equals $g\log r^D(x) - g\log r^{\wt D}(x)$ by Lemma~\ref{lemma-3.3} and~\eqref{E:3.3}. Since $\alpha^2=4/g$, it follows that
	\begin{equation}\label{e:r-shift}
		\texte^{-\frac12\alpha^2\Var(h^D_t(x))} r^D(x)^2 = \rme^{-\frac12\alpha^2\Var(h^{\wt D}_t(x))} r^{\wt D}(x)^{2} \,.
	\end{equation}
	 From the definition~\eqref{E:2.7} of~$Z^D_t$ and the pointwise equality \eqref{E:4.23w} we then readily get 
	\begin{equation}
		\label{e:4.26}
		\int f(x) Z^D_t(\rmd x) = \int f(x) \rme^{\alpha \Phi^{D,\wt D}(x)} Z^{\wt D}_t(\rmd x) \,,
	\end{equation}
	for all $f \in C_\cc(\wt{D})$ and~$t$  so large that~$\texte^{-t}$  becomes smaller than  the distance between the support of~$f$ and the complement of~$\wt D$.   Taking  $t\to\infty$  along which the weak convergence of~$Z_t^D$ and~$Z_t^{\wt{D}}$ to~$Z^D$ and, resp.,~$Z^{\wt{D}}$ holds almost-surely, we may pass to the limit in~\eqref{e:4.26}. This shows that,  under the coupling \eqref{e:shift-GM}, 
	\begin{equation}
		Z^D = \rme^{\alpha \Phi^{D,\wt D}} Z^{\wt D}\quad\text{a.s.}
	\end{equation}
	 as measures on~$\wt D$. 
	Since the measures on both sides do not charge $D \setminus \wt{D}$, 
	 this holds also in the sense of measures on~$D$. From \eqref{e:shift-GM}  we then  obtain
	\begin{equation}
		\bigl(\texte^{\alpha\Phi^{D,\wt D}} M^{\wt D},\texte^{\alpha\Phi^{D,\wt D}} Z^{\wt D}\bigr)  \,\,\,\laweq\,\,\,\bigl(M^D,Z^D).
	\end{equation}
	Theorem~\ref{thm-3.10} now implies \eqref{E:4.20io}.
\end{proof}

If it were not for the reliance on the conditions of Theorem~\ref{BL2-thm2.8}, Theorem~\ref{thm-4.3} would run very close to the characterization of  the  \emph{subcritical} Gaussian Multiplicative Chaos measures by their behavior under Cameron-Martin shifts established by Shamov~\cite{Shamov}. In order to get the critical case aligned with subcritical ones, one would have to prove the conditions of Theorem~\ref{BL2-thm2.8} directly from \eqref{E:4.19iu}, a task we will consider returning to in future work.

\section{Moment calculations and estimates}
\label{sec4}\noindent
We are ready to move to the proof of our main results.
In this section we perform moment calculations for the  near-extremal measure~$\zeta^D$; these will  feed directly into the proofs of Theorem~\ref{thm-A}. Throughout, we will assume that a domain~$D\in\mathfrak D$ and a set of approximating lattice domains $\{D_N\}_{N\ge1}$ have been fixed. 

\subsection{Strategy and key lemmas}
In order to motivate the forthcoming derivations, let us first discuss the main steps of the proof of Theorem~\ref{thm-A}. Focussing only on the limit of the measures~$\zeta_N^D$, our aim is to show convergence in law of $\langle\zeta^D_N,f\rangle$ for a class of test functions~$f$. We will proceed using a similar strategy as in~\cite{BL1,BL2,BL3,BL4}: First prove tightness, which permits extraction of subsequential limits, and then identify the limit uniquely by its properties. 

The tightness will be resolved by a first-moment calculation. However, since~$Z^D$ is known to lack the first moment and the convergence of $\langle\zeta^D_N,f\rangle$ cannot thus hold in the mean, the moment calculation must be restricted by a suitable truncation. In order to identify the subsequential limits with the expression on the right of \eqref{E:2.10}, we invoke Theorem~\ref{BL2-thm2.8}. This requires checking  the conditions in the statement of which the most subtle
is the uniform Laplace-transform tail \eqref{E:1.25a}. It is here where we need the present version  (of Theorem~\ref{BL2-thm2.8})  as opposed to \cite[Theorem~2.8]{BL2} because, due to our reliance on subsequential convergence, we do not have a direct argument for dilation covariance of the limit  measure. 

Both the tightness and the uniform Laplace-transform tail will be inferred from the following two propositions whose proof constitutes the bulk of this section:

\begin{proposition}[Truncated first moment]
	\label{lem:first-moment}
	For all bounded continuous~$f\colon D \times \bbR\to\R$,
	\begin{equation}
		\label{eq:1m-as}
		\lim_{u\to\infty} \limsup_{N\to\infty}
		\biggl|\frac1u\bbE \Bigl(\la\zeta^D_N\,,f \ra\,;\; h^{D_N}\leq m_N+u \Bigr) 
		-\frac{2\texte^{2c_0/g}}{g\sqrt{2\pi g}} \int_{D\times(0,\infty)}\!\!\!\!\! s\; \rme^{-\frac{s^2}{2g}}\,
		r^D(x)^2 f(x,s)\;\rmd x \,\rmd s\biggr|=0.
	\end{equation}
	Moreover, there exists a constant $c'=c'(D)<\infty$ such that
	\begin{equation}
		\label{eq:1m-ub}
		\bbE \Bigl(\zeta^D_N(D\times\bbR);\, h^{D_N}\leq m_N+u \Bigr)\leq c'u
	\end{equation}
	for all $N\ge1$ and $u\ge1$.
\end{proposition}

\begin{proposition}[Truncated second moment]
	\label{lem:second-moment}
	For each $\delta>0$ there exists a continuous function $\theta\colon[0,\infty)\to[0,\infty)$ with $\lim_{u\to\infty} \theta(u)/u=0$ such that
	\begin{equation}
		\label{eq:2m}
		\limsup_{N\to\infty}\bbE\Bigl(\zeta^D_N\bigl(D^\delta\times(-\infty,\delta^{-1}]\bigr)^2 ;\; h^{D_N}\leq m_N+u\Bigr)
		\leq \rme^{\alpha u +\theta(u)}
	\end{equation}
	for all $u\ge0$.
\end{proposition}

Before delving into the proof, let us see how tightness follows from this:

\begin{lemma}[Tightness]
	\label{lem:tightness}
	The family $\{\zeta^D_N(D\times\bbR)\colon N\ge1\}$ of $\R$-valued random variables is tight. Consequently, the measures $\{\zeta^D_N\colon N\ge1\}$ are tight relative to the vague topology on~$\overline D\times\overline\R$.
\end{lemma}

\begin{proof}
	Since $\zeta^D_N$ is supported on~$D\times\bbR$, the second part of the claim follows from the first and so we just need to show that $\{\zeta^D_N(D\times\bbR)\colon N\ge1\}$ is tight.  Let $u>0$ and~$c > 0$. By the union bound and Markov's inequality,
	\begin{equation}
		\bbP\bigl(\zeta^D_N(D\times \bbR)>a\bigr) 
		\leq \bbP\Bigl(\,\max_{x\in D_N} h^{D_N}>m_N+u\Bigr)+\frac{1}{a}\bbE\Bigl(\zeta^D_N(D\times \bbR);\,h^{D_N}\leq m_N+u\Bigr) \,.
	\end{equation}
	Lemma~\ref{lemma-DGFF-tail} and Proposition~\ref{lem:first-moment} dominate the right-hand side by a quantity of order $u\texte^{-\alpha u}+a^{-1}u$, uniformly in~$N\ge1$. This can be made arbitrary small by choosing, e.g., $a:=u^2$ and taking~$u$ sufficiently large.
\end{proof}

The proof of the Laplace transform asymptotic is more involved. The bulk of the computation is the content of  the next lemma. Here and in the following, we use the notation $f_1\otimes f_2(x,s):=f_1(x)f_2(s)$.

\begin{lemma}[Laplace transform asymptotic]
	\label{lem:tilted}
	For all bounded continuous $f_1\colon D\to[0,\infty)$ and nonnegative $f_2\in C_\cc(\R)$,
	\begin{equation}
		\label{e:2.3}
		\lim_{\lambda\downarrow 0} \limsup_{N\to\infty}\Biggl|
		\frac{\bbE\left(
			\big\la\zeta^D_N\,,f_1 \otimes f_2 \big\ra
			\rme^{-\lambda\la\zeta^D_N\,,\Ind_D \otimes f_2 \ra}\right)}{\log(\ffrac1\lambda)}
		-c_\star\int_{D\times(0,\infty)}\!\! s\; \rme^{-\frac{s^2}{2g}}\,
		r^D(x)^2 f_1(x)f_2(s)\;\rmd x \,\rmd s\Biggr|=0\,,
	\end{equation}
	where $c_\star$ is the constant from Theorem~\ref{thm-A}.
\end{lemma}

\begin{proof}
	We will separately show that the \emph{limes superior}, resp., \emph{limes inferior} as $N\to\infty$ and $\lambda\downarrow0$ of the ratio is bounded from above, resp., below by the integral expression in \eqref{e:2.3}. For the \emph{limes superior} we parametrize the~$\lambda\downarrow0$ limit as the $u\to\infty$ limit of $\lambda := u\rme^{-\alpha u}$. For this choice of~$\lambda$, the expectation in \eqref{e:2.3} is bounded from above by
	\begin{equation}
		\bbE\Bigl(\big\la\zeta^D_N,\,f_1\otimes f_2\big\ra\rme^{-u\rme^{-\alpha u}\la\zeta^D_N,\,\Ind_D\otimes f_2\ra};\,\max_{x\in D_N} h^{D_N}_x>m_N+u\Bigr)
		+ \bbE\Bigl(\big\la\zeta^D_N,\,f_1\otimes f_2\big\ra;h^{D_N}\leq m_N+u\Bigr).
		\label{e:2.4}
	\end{equation} 
	Using that $\max_{t\geq 0} t\rme^{-\lambda t}=(\lambda \rme)^{-1}$  for the above $\lambda$, the first term is at most 
	\begin{equation}
		\|f_1\|_\infty\, u^{-1}\rme^{\alpha u -1}\,\bbP\Bigl(\,\max_{x\in D_N} h^{D_N}_x> m_N+u\Bigr),
	\end{equation}
	which upon division by~$\log(\ffrac1\lambda)=\alpha u+\log(1/u)$ tends to zero in the limits as $N \to \infty$ followed by $u \to \infty$, thanks to~\eqref{e:A1}. By Proposition~\ref{lem:first-moment}, the second term in~\eqref{e:2.4} divided by $\alpha u$ tends to the integral term in~\eqref{e:2.3} under these limits.
	
	For the \emph{limes inferior}, pick~$\delta>0$ and parametrize the $\lambda\downarrow0$ limit as the $u\to\infty$ limit of $\lambda :=\rme^{-\alpha u-\theta(u)}$, for~$\theta$ as in Proposition~\ref{lem:second-moment}. (This exhausts all small values of~$\lambda$ because~$\theta$ is continuous with $\theta(u)=o(u)$.) For this choice of~$\lambda$, the expectation in~\eqref{e:2.3} is at least
	\begin{equation}
	\begin{aligned}
		\bbE\Bigl(\bigl\la\zeta^D_N\,,f_1\otimes f_2\bigr\ra;&h^{D_N}\leq m_N + u\Bigr)\\
		&-\bbE\biggl(\big\la\zeta^D_N\,,f_1\otimes f_2\big\ra\Bigl(1-\exp\Big\{-\rme^{-\alpha u-\theta(u)}\big\la\zeta^D_N\,,\Ind_{D}\otimes f_2\big\ra\Big\}\Bigr);h^{D_N}\leq m_N + u\biggr).
	\end{aligned}
	\end{equation}
	By Proposition~\ref{lem:first-moment}, the first term in the last expression divided by $\log(1/\lambda)=\alpha u+o(u)$ converges to the desired limit as $N\to\infty$ and then $u\to\infty$. The second term (without the minus sign) is at most $\|f_1\|_\infty$ times
	\begin{equation}
	\begin{aligned}
		\label{E:4.8e}
		\bbE \Bigl (\bigl\la\zeta^D_N\,,\Ind_{D\smallsetminus D^\delta}\otimes f_2\bigr\ra;h^{D_N}\leq &m_N+u\Bigr)
		+ (1-\rme^{-\delta})\bbE \Bigl(\big\la\zeta^D_N\,,\Ind_{D^\delta}\otimes f_2\big\ra;h^{D_N}\leq m_N+u\Bigr)
		\\
		&+ \bbE\Bigl(\big\la\zeta^D_N\,,\Ind_{D^\delta}\otimes f_2\big\ra;h^{D_N}\leq m_N+u,\,\big\la\zeta^D_N\,,\,\Ind_{D}\otimes f_2\big\ra>\delta\rme^{\alpha u+\theta(u)}\Bigr).
	\end{aligned}
	\end{equation}
	Proposition~\ref{lem:first-moment} ensures that, upon division by~$u$, the \emph{limes superior} as $N \to \infty$ and $u \to \infty$ of first two terms is at most order $\Leb(D\smallsetminus D^\delta)+\delta$.
	Using the Cauchy-Schwarz and Markov inequalities, the third term is in turn at most the square root of 
	\begin{equation}
		\bbE\left(\big\la\zeta^D_N\,,\Ind_{D^\delta}\otimes f_2\big\ra^2;h^{D_N}\leq m_N+u\right)
		\delta^{-1}\rme^{-\alpha u-\theta(u)} \bbE\Bigl(\big\la\zeta^D_N\,,\Ind_{D}\otimes f_2\big\ra;h^{D_N}\leq m_N+ u\Bigr).
	\end{equation}
	Thanks to our restriction on the support of~$f_2$, Proposition~\ref{lem:second-moment} bounds the  first expectation by a constant times $\rme^{\alpha u+\theta(u)}$ while Proposition~\ref{lem:first-moment} shows that the  second  expectation is at most~$O(u)$. The last expectation in \eqref{E:4.8e} is thus at most $O(\sqrt u)$ which upon dividing by~$u$ vanishes in the limits $N\to\infty$ and $u\to\infty$. 
	
	We conclude that the \emph{limes inferior} of the expectation term in \eqref{e:2.3} is at least the integral term minus a quantity of order $\Leb(D\smallsetminus D^\delta)+\delta$. As $D^\delta\uparrow D$ implies $\Leb(D\smallsetminus D^\delta)\downarrow0$ as~$\delta\downarrow0$, the claim follows.
\end{proof}

\subsection{Truncated first moment}
The remainder of this section is devoted to the proof of the above two key lemmas. When dealing with various Gaussian densities, we will frequently refer to:

\begin{lemma}
	\label{lem:Gaussian}
	Denote $v_N(y):=g\log N + gy$. Then for all $c\ge1$ and all $\ep>0$:
	\settowidth{\leftmargini}{(11)}
	\begin{enumerate}
		\item
		For all $N >  \rme^c$, all $y\in[-c,c]$ and all~$s$ with $\abs{s}\leq(\log N)^{1-\ep}$,
		\begin{equation}
			\label{item:as:Gaussian}
			\exp\left(-\frac{(m_N-s)^2}{2v_N(y)}\right)
			=N^{-2}(\log N)^{3/2}\rme^{\alpha s}
			\exp\left(-\frac{s^2}{2g\log N}+2y+\delta_N(s,y)\right)\,,
		\end{equation}
		where $\delta_N(s,y)\in \bbR$ satisfies
		\begin{equation}
			\lim_{N\to\infty}\,\,\sup_{|s|\leq(\log N)^{1-\ep}}\,\sup_{|y|\le c}\,\bigl|\delta_N(s,y)\bigr|=0\,.
		\end{equation}
		\item
		There exist constants $c_1, c_2<\infty$ such that
		\begin{equation}
			\label{item:ub-r:Gaussian} 
			\begin{aligned}
				\exp\left(-\frac{(m_N+r)^2}{2v_N(y)}\right)
				&\leq c_1N^{-2}v_N(y)^{3/2}\rme^{-\alpha r}
				\exp\left(-\frac{r^2}{2v_N(y)}+\abs{r}\frac{3\alpha\log\log N}{8\log N}\right)\\
				&\leq c_2N^{-2}v_N(y)^{3/2}\rme^{-\alpha r}
			\end{aligned}
		\end{equation}
		for all $N >  \texte^c$, all $y\in[-c,c]$ and all $r\in\bbR$.
	\end{enumerate}
\end{lemma}

\begin{proof}
	Abbreviate $n:=\log N$. Then (for~$n\ge1$) $m_N=2\sqrt g n-\frac34\sqrt g\log n$ and so
	\begin{equation}
		\label{eq:m_kappa-r}
		-\tfrac{1}{2}(m_N+r)^2
		= -2g n^2 +\tfrac{3}{2}gn\log n-2\sqrt{g}\,rn
		-\tfrac{9}{32}g\log(n)^2+\tfrac{3}{4}\sqrt{g}\,r\log n-\tfrac12r^2.
	\end{equation}
	From $v_N(y)=gn+gy$ we get
	\begin{equation}
		\label{eq:var-Taylor}
		\frac{1}{v_N(y)}
		=\frac{1}{gn}-\frac{y}{gn^2+gny}=\frac{1}{gn}-\frac{1}{gn^2}\left(y+o(1)\right),
	\end{equation}
	where $o(1)\to0$ as $N\to\infty$ uniformly in $y\in[-c,c]$.
	Then~\eqref{item:as:Gaussian} follows by combining~\twoeqref{eq:m_kappa-r}{eq:var-Taylor} and noting that~$|s|\le n^{1-\epsilon}$. For~\eqref{item:ub-r:Gaussian} we bound the second to last term in \eqref{eq:m_kappa-r} by its absolute value and then use that $1/v_N(y)\le1/(gn-c)$. For the last inequality in \eqref{item:ub-r:Gaussian} we also need that the minimum of $x\mapsto Ax^2+Bx$ is, for~$A>0$, equal to $-\frac14 B^2/A$.
\end{proof}

We start by the first moment asymptotics:

\begin{proof}[Proof of Proposition~\ref{lem:first-moment}]
	Let us again use the shorthand $n:=\log N$ whenever convenient.
	We first show \eqref{eq:1m-as} for~$f$ with compact support in $D\times\mathbb R$ as this illustrates the main computation driving the proof.
	For any~$u\ge1$, $x\in D$ and~$r\in\R$, let
	\begin{equation}
		Q_{N,x}(u,r):=\bbP\Bigl(h^{D_N}\leq m_N +u\,\Big|\,h^{D_N}_{\lfloor Nx\rfloor}=m_N-r\Bigr).
	\end{equation}
	The expectation in~\eqref{eq:1m-as} then equals
	\begin{equation}
		\label{eq:first-moment}
		\frac{1}{n}\sum_{x\in D_N}\int_\R
		f(x/N,s)
		\,\rme^{-\alpha s\sqrt{n}}\,Q_{N, x/N }(u,s\sqrt n)\,
		\bbP\biggl(\frac{m_N-h^{D_N}_x}{\sqrt n}\in\rmd s\biggr).
	\end{equation}
	By our restriction to~$f$ with compact support, there is $\delta>0$ such that $\rmd_\infty(x, (D_N)^\cc)>\delta N$ holds for all~$x$ effectively contributing to the sum. Using \eqref{item:as:Gaussian} in conjunction with the Green function asymptotic \eqref{eq:Green}, for the integrating measure in \eqref{eq:first-moment} we get 
	\begin{equation}
		\rme^{-\alpha s\sqrt{n}}\,\bbP\biggl(\frac{m_N-h^{D_N}_x}{\sqrt n}\in\rmd s\biggr)
		=\frac1{\sqrt{2\pi g}}\,N^{-2}n^{3/2}\,r^D(x/N)^2\,\texte^{2c_0/g+o(1)}\,\texte^{-\frac{s^2}{2g}}\textd s,
	\end{equation}
	where~$o(1)\to0$ as~$N\to\infty$ uniformly in $(x,s)\in\supp f$. This
	permits us to bring the $o(1)$ term outside the integral which  turns \eqref{eq:first-moment} into
	\begin{equation}
		\label{eq:bd-Gre-bin}
		\begin{aligned}
		\frac{\texte^{o(1)+2c_0/g}}{\sqrt{2\pi g}}\frac1{N^2}\sum_{x\in D_N}r^D(&x/N)^2 \int_\R \,f(x/N,s)\sqrt n\,Q_{N,x/N}(u,s\sqrt n)\,\texte^{-\frac{s^2}{2g}}\,\textd s
		\\
		&= \frac{\texte^{o(1)+2c_0/g}}{\sqrt{2\pi g}}\int_{D\times\R} r^D(x)^2 \,f_N(x,s)\sqrt n\,Q_{N,x}(u,s\sqrt n)\,\texte^{-\frac{s^2}{2g}}\,\textd x\,\textd s,
		\end{aligned}
	\end{equation}
	where $f_N(x,s):=f(\lfloor Nx\rfloor/N,s)$  and where we noted  that $Q_{N,x}=Q_{N,\lfloor Nx\rfloor/N}$. Henceforth we assume that~$N$ is so large that \eqref{E:2.1} ensures that $f_N(x,s)\ne0$ implies~$x\in D^{\delta/2}$.
	
	Concerning the integrand, the second part of Lemma~\ref{prop:UB} bounds $Q_{N, x}(u,r)$ by
	\begin{equation}
		\label{E:4.20}
		Q_{N,x}(u,r)\le c\,\frac{(1+u)(1+(u+r)^+)}{n}
	\end{equation}
	for all $x \in D^{\delta/2}$, and so $\sqrt n Q_{N,x}(u,s\sqrt n)$ is bounded by a constant times $(1+u)(1+s^+)$ uniformly in $(x,s)\in\supp(f_N)$ and~$n\ge u^2$. For an asymptotic expression, we first note that $Q_{N,x}(u,s\sqrt{n})=0$ whenever $s\sqrt{n} < -u$ for all $x \in D$. On the other hand, if $s \geq 0$, then by Lemma~\ref{prop:asymptotic}, as $N \to \infty$, 
	\begin{equation}
		Q_{N,x}(u,s\sqrt n)=\bigl[2+o(1)\bigr]\,\dfrac{\cL^{x}(u)\cR(u+s\sqrt n)}{gn}.
	\end{equation}
	As $\cR$ is asymptotic to the identity function at large values of its argument, we get 
	\begin{equation}
		\sqrt n\,Q_{N,x}(u,s\sqrt n) \,\,\underset{N\to\infty}\longrightarrow\,\, \frac2g\LL^{x}(u)\,s\Ind_{[0,\infty)}(s) \,,
	\end{equation}
	for all $x \in D$ and $s \in \bbR$.
	
	Since the assumptions on $f$ ensure that $f_N(x,s)\to f(x,s)$, the Dominated Convergence Theorem shows that \eqref{eq:first-moment} converges as $N \to \infty$ to 
	\begin{equation}
		\label{e:4.21}
		\texte^{2c_0/g}\frac2{g\sqrt{2\pi g}}\int_{D\times[0,\infty)} r^D(x)^2 f(x,s) \LL^{x}(u)\,s\texte^{-\frac{s^2}{2g}}\,
		\textd x\,\textd s.
	\end{equation}
	Dividing by $u$, using~\eqref{e:lem:rw-as:fg} (which holds uniformly whenever the integrand is non-zero) and applying the Dominated Convergence Theorem then yield the claim for all~$f$ with compact support.
	
	To include $f$ whose support is of the form $A\times\bbR$ for a compact set $A\subseteq D$,  it remains to consider the expression in~\eqref{eq:first-moment} with the integral therein restricted to $s\in[M,\infty)$.
	With this restriction of the integral, and bounding~$f$ by $\|f\|_\infty \Ind_A$, display~\eqref{eq:first-moment} is at most a constant times
	\begin{equation}
		\int_{A\times[M,\infty)}\,\sqrt n\,Q_{N,x}(u,s\sqrt n)\,\texte^{-\frac{s^2}{2g + c/ n}+s\frac{3\alpha\log n}{8\sqrt{n}}}\,\textd x\,\textd s \,.
	\end{equation}
	Here we used \eqref{item:ub-r:Gaussian} in place of \eqref{item:as:Gaussian}, the constant $c>0$ again arises from the Green function asymptotics~\eqref{eq:Green}, and we bounded $r^D(x)^2$ by a constant using its definition~\eqref{E:3.3} and the fact that $A\subseteq D$ is compact. Bounding also $Q_{N,x}(u,s\sqrt{n})$ by~\eqref{E:4.20} and using that~\eqref{eq:first-moment} equals the expectation in~\eqref{eq:1m-as},  for each~$\ep>0$  we obtain
	\begin{equation}
		\label{e:5.24}
		\frac{1}{1+u}\bbE\Bigl(\zeta^D_N(A\times[M,\infty));\, h^{D_N}\leq m_N + u\Bigr)
		\leq c'\int_{s\in[M,\infty)}\rmd s\,  \rme^{-s^2(\frac{1}{2g}-\ep)}\,,
	\end{equation}
	 once~$N$ is sufficiently large, where~$c'$ is a constant that may depend on~$\ep$.  The right-hand side vanishes in the limit as $N\to\infty$, $u\to\infty$ and $M\to\infty$, taken in this order.
	
	Finally, to include functions with arbitrary support, we need:
	\begin{lemma}
		\label{lemma-tight}
		For each~$D\in\mathfrak D$ there is $c\in(0,\infty)$ such that for all measurable~$A\subseteq D$, all $u\ge1$ and all $N\ge3$,
		\begin{equation}
			\label{eq:first-moment-bd}
			u^{-1}\,\bbE\Bigl(\zeta^D_N(A\times\bbR);\,h^{D_N}\leq m_N+u\Bigr)\leq c \frac1{N^2}\bigl|\{x\in\Z^2\colon x/N\in A\}\bigr|.
		\end{equation}
	\end{lemma}
	
	\noindent
	Postponing the proof for a moment, \eqref{eq:first-moment-bd} shows that restricting the support of~$f$ to $D^\delta$ with the help of a bounded mollifier that differs from one or zero only   in  $D^\delta\smallsetminus D^{2\delta}$, the expectation in \eqref{eq:1m-as} changes by at most a  constant times  $u N^{-2} |D_N\smallsetminus D_N^{2\delta}|$. The integral in \eqref{eq:1m-as} also changes by a  constant times  $\leb(D\smallsetminus D^{2\delta})$.
	 The discretization~\eqref{E:2.1} ensures that the elements of $x\in D_N$ can be placed in disjoint unit open squares that are all  contained in~$ND$. Hence, 
	\begin{equation}
		N^{-2}|D_N\smallsetminus D_N^\delta|\le \leb(D\smallsetminus D^{2\delta}).
	\end{equation}
	 Thus, using also that $\leb(D\smallsetminus D^{2\delta})\to0$ as~$\delta\downarrow0$, the errors in the  above  approximation tend to zero  and the statement holds for functions with arbitrary support. 
\end{proof}

It remains to give: 

\begin{proof}[Proof of Lemma~\ref{lemma-tight}]
	Fix~$x\in D_N$ and let
	$B:=\{z\in D_N:\: \rmd_\infty(x,z)<\rmd_\infty(x, (D_N)^\cc)\}$,
	$B':=\{z\in D_N:\: \rmd_\infty(x,z)<\tfrac12 \rmd_\infty(x, (D_N)^\cc)\}$.
	Using the Gibbs-Markov property (Lemma~\ref{lem:Gibbs-Markov}) to write $h^{D_N}=h^B+\varphi^{D_N,B}$ we then get
	\begin{equation}\label{E:4.26}
		\begin{aligned}
			\bbE\Bigl(\rme^{\alpha (h^{D_N}_x-m_N)} ;\,&h^{D_N}\leq m_N+u\Bigr)
			\le \bbE\Bigl(\rme^{\alpha (h^{D_N}_x-m_N)}  ;\,h^{D_N}\leq m_N+u\text{ in }B'\Bigr)
			\\
			&\leq\sum^\infty_{k= 1}\bbP\Bigl(\,\max_{z\in B'}\bigl|\varphi^{D_N,B}_z\bigr|\geq k -1  \Bigr)
			\,\rme^{\alpha(k+1)}\bbE\Bigl(\rme^{\alpha(h^{B}_x-m_N)}  ;\,h^B\leq m_N+u+k\text{ in }B'\Bigr).
		\end{aligned}
	\end{equation}
	 Set $L:=\rmd_\infty(x,(D_N)^\cc)$. Since~$x$ is at the center of~$B$, the calculation leading to~\eqref{e:5.24} applies (using Lemma~\ref{prop:UB} with $D:=B(0,1)$, $N:=L+1$ and $\delta:=\tfrac12$). This bounds the expectation on  the right-hand side by a constant times
	\begin{equation}
		(m_N-m_L+u+k)\frac{1\vee\log L}{L^2}\texte^{\alpha m_L - \alpha m_N} \,,
	\end{equation}
	uniformly in~$x\in D_N$ and~$u\ge1$.
	Using that (for $N\ge3$) 
	\begin{equation}
		\frac{1\vee\log L}{L^2}\texte^{\alpha m_L}
		=\frac{L^2}{N^2}\sqrt{\frac{\log N}{1\vee\log L}}\frac{\log N}{N^2}\texte^{\alpha m_N}
	\end{equation}
	and that $m_N-m_L=O(1)+O(\log(N/L))$ and $L=O(N)$ uniformly in~$x\in D_N$, we conclude
	\begin{equation}
		\label{E:4.31}
		\bbE\Bigl(\rme^{\alpha(h^{B}_x-m_N)} ;\,h^B\leq m_N+u+k\text{ in }B'\Bigr)
		\le c(u+k)\frac{\log N}{N^2}
	\end{equation}
	for some constant~$c\in(0,\infty)$, uniformly in~$x\in D_N$,~$k\ge0$ and~$u\ge1$.
	
	Thanks to Lemma~\ref{lemma-3.2a}, $\max_{z\in B'}\Var(\varphi^{D_N,B}(z))$ is bounded uniformly in (the centering point)~$x\in D_N$ and~$N\ge1$. A standard argument based on the Fernique estimate along with Borell-TIS inequality then shows that, for some~$\tilde c>0$,
	\begin{equation}
		\bbP\Bigl(\,\max_{z\in B'}\bigl|\varphi^{D_N,B}_z\bigr|\geq k  -1 \Bigr)\le\texte^{-\tilde c k^2},\quad k\ge  1. 
	\end{equation}
	Plugging this, along with \eqref{E:4.31}, into \eqref{E:4.26}, the sum over~$k\ge  1$ may be performed resulting in an expression bounded by a constant times $u(\log N)N^{-2}$, uniformly in~$u\ge1$ and $x\in D_N$. The claim follows by summing over~$x\in D_N$ with $x/N\in A$ and invoking the definition of~$\zeta^D_N$.
\end{proof}
We remark that it is exactly Lemma~\ref{lemma-tight} that forces us to restrict attention to discretization via \eqref{E:2.1}. This is in turn caused by our reliance on Lemma~\ref{lemma-3.2a} which only applies when the ``maximal'' discretization \eqref{E:2.1} is used.

\subsection{Truncated second moment}
\label{sec:proof:second-moment}\noindent
Our next task is the proof of the upper bound in Proposition~\ref{lem:second-moment} on the truncated second moment of measures~$\zeta^D_N$. Given an integer $N\ge1$, distinct vertices $x,y\in D_N$ and a real number~$u\ge1$, let $f_{u}^{N;x,y}\colon\R^2\to[0,\infty)$ denote the density in
\begin{equation}
	\label{E:4.34}
	\bbP\Bigl(h^{D_N}\leq m_N +u,\,\,h^{D_N}_x- m_N\in \textd s,\, h^{D_N}_y-m_N\in\textd t\Bigr) = f_{u}^{N;x,y}(s,t)\textd s\,\textd t.
\end{equation}
In light of the explicit form of the law of $(h_x^{D_N},h_y^{D_N})$, we may and will assume that $f_{u}^{N;x,y}$ is left continuous in each variable which fixes this function uniquely.

Abbreviate~$M:=\delta^{-1}$. Then, thanks to the restriction on the maximum of~$h^{D_N}$,
\begin{equation}
\begin{aligned}
	\label{eq:second-int}
	\bbE\Bigl(\bigl(\zeta^D_N(&D^\delta\times(-\infty, M])\bigr)^2;\,h^{D_N}\leq m_N+u\Bigr)
	\\
	&=\frac{1}{(\log N)^2}\sum_{x\in D^\delta_N}
	\int_{[-M\sqrt{\log N},u]}
	\,\rme^{2\alpha s }\,\bbP\bigl(h^{D_N}_x- m_N\in\rmd s, h^{D_N}\leq m_N +u\bigr)
	\\
	&\qquad\qquad\qquad+\frac{1}{(\log N)^2}\sum_{\begin{subarray}{c}
			x,y\in D^\delta_N\\x\ne y
	\end{subarray}}
	\int_{[-M\sqrt{\log N},u]^2}
	\,\rme^{\alpha (s+t)}\,f_{u}^{N;x,y}(s,t)\,\textd s\, \textd t.
	\quad
\end{aligned}
\end{equation}
To control the second integral  we will need a good estimate on $f_{u}^{N;x,y}(s,t)$:

\begin{lemma}[2-density estimate]
	\label{lem:2-density}
	Let $\ell_N(x,y)$ be the largest integer for which the $\ell^\infty$-balls $B(x,\rme^\ell+1)$ and $B(y,\rme^\ell+1)$ are disjoint. 
	For each $\delta,\ep>0$ there exist $C,p\ge1$ such that, with $\ell:=\ell_N(x,y)$ and $n:=\log N$,
	\begin{equation}
		\rme^{\alpha(s+t)}\,f_{u}^{N;x,y}(s,t)\le CN^{-2}\rme^{-2\ell}\left(\frac{n}{\ell(n-\ell)}\right)^{3/2}
		(1+s^-)(1+t^-)\rme^{(\alpha+\epsilon)u}
	\end{equation}
	holds for all $N\ge 3$, all $x,y\in D^\delta_N$ with $1\leq \ell_N(x,y)\leq n - p$, all $s,t\in\bbR$ and all $u\geq 1$.
\end{lemma}

Before we delve into the proof, let us show how this implies:

\begin{proof}[Proof of Proposition~\ref{lem:second-moment}]
	Fix $\delta,\ep>0$ and let $p$ be as in Lemma~\ref{lem:2-density}.
	Abbreviate $n:=\log N$ and $M:=\delta^{-1}$.
	We start by bounding the short-range part of the double sum in \eqref{eq:second-int} using that $\texte^{\alpha(s+t)}\le\texte^{\alpha s+\alpha u}$ on the integration domain. This~gives
	\begin{equation}
	\begin{aligned}
		\label{E:4.37}
		\frac{1}{(\log N)^2}\sum_{x\in D^\delta_N}
		&\int_{[-M\sqrt{\log N},u]}
		\,\rme^{2\alpha s }\,\bbP\bigl(h^{D_N}_x- m_N\in\rmd s, h^{D_N}\leq m_N +u\bigr)
		\\
		&+
		\frac{1}{(\log N)^2}\sum_{\substack{x,y\in D^\delta_N\\ \ell_N(x,y)< 1\\ x\ne y}}
		\int_{[-M\sqrt{\log N},u]^2}
		\rme^{\alpha (s+t)}\,f_{u}^{N;x,y}(s,t)\,\textd s\,\textd t
		\qquad\qquad
		\\
		&\qquad\qquad\qquad\le\frac1{n^2}\sum_{\substack{x,y\in D^\delta_N\\ \ell_N(x,y)< 1}}\texte^{\alpha u}\,
		\int
		\rme^{\alpha s}\,\bbP\bigl(h^{D_N}_x- m_N\in\rmd s, h^{D_N}\leq m_N +u\bigr),
	\end{aligned}
	\end{equation}
	 where we also used definition~\eqref{E:4.34} of $f^{N;x,y}_u(s,t)$ and integrated over $t$. 
	Since $\ell_N(x,y)< 1$ implies $y\in B(x, 2\texte^{1} +2)$, using $|B(x,r)\cap D^\delta_N|\le 4r^2$ we can estimate the right-hand side of \eqref{E:4.37} by
	\begin{equation}
		\label{E:4.38}
		\bbE\left(\zeta^D_N(D^\delta\times\R);h^{D_N}\leq m_N+u\right)
		\frac{ 64 \, \rme^{2+\alpha u}}{n}.
	\end{equation}
	Lemma~\ref{lemma-tight} ensures the expectation is bounded uniformly in~$N\ge1$. Taking $N\to\infty$, \eqref{E:4.38} tends to zero.
	
	We will handle the remaining portion of the sum in \eqref{eq:second-int} by a covering argument. For this let us consider a general measurable set $A\subseteq D^\delta$ with $\diam(A)<\delta\rme^{-p}$ and denote $A_N:=\{x\in D_N\colon x/N\in A\}$.
	Then, for~$\delta$ sufficiently small, the definition of $\ell_N(x,y)$ ensures $\ell_N(x,y)\leq n-p$ for all $x,y\in A_N$. Lemma~\ref{lem:2-density} then gives, for $N\geq N_0$,
	\begin{equation}
	\begin{aligned}
		\label{E:4.39a}
		&\frac{1}{(\log N)^2}\sum_{\begin{subarray}{c}
				x,y\in A_N\\ \ell_N(x,y)\ge 1
		\end{subarray}}
		\int_{[-M\sqrt{\log N},u]^2}
		\,\rme^{\alpha (s+t)}\,f_{u}^{N;x,y}(s,t)\,\textd s\, \textd t
		\\
		&\qquad\le
		C\frac1{N^2(\log N)^2}\sum_{\ell=1}^{n-p}
		\sum_{\substack{x,y\in A_N\\\ell_N(x,y)=\ell}}
		\rme^{-2\ell}
		\left(\frac{n}{\ell(n-\ell)}\right)^{3/2}
		\int_{[-M\sqrt{\log N},u]^2}
		(1+s^-)(1+t^-)\,\rme^{(\alpha+\epsilon)u}\rmd s\,\rmd t\,,
	\end{aligned}
	\end{equation}
	 where the constant $C>0$ depends on $\delta$, $\ep$ and $D$. For $N$ with $M\sqrt{\log N}\ge u$, the integral on the right-hand side  is at most order~$M^4(\log N)^2\rme^{(\alpha+\epsilon)u}$.  By this bound, and as  for each~$x\in D_N$ there are at most a constant times $\texte^{2\ell}$ vertices~$y$ with~$\ell_N(x,y)=\ell$, the expression  on the right-hand side of \eqref{E:4.39a} is  at most a constant times
	\begin{equation}
		\label{E:4.40a}
		M^4\rme^{(\alpha+\epsilon)u}\frac1{N^2}|A_N|\sum_{\ell=1}^{n-p}\left(\frac{n}{\ell(n-\ell)}\right)^{3/2}
		\le 2^{5/2}\,M^4 \rme^{(\alpha+\epsilon)u}\frac1{N^2}|A_N|\sum_{\ell\ge 1}\ell^{-3/2}.
		\qquad
	\end{equation}
	All the above bounds are uniform in the choice of the set~$A$.  
	
We  now  cover~$D^\delta$ by $K=O(\delta^{-2}\texte^{2p})$ of $\ell^\infty$-balls $A^{(1)},\dots,A^{(K)}$ of radius~$\delta\texte^{-p}$ (recall that $p$ and hence $K$ depend on $\delta$, $\ep$ and $D$). Using  
	\begin{equation}
		\bbE\Bigl(\bigl(\zeta^D_N(D^\delta\times(-\infty, M])\bigr)^2;\,h^{D_N}\leq m_N+u\Bigr)
		\le K\sum_{i=1}^K\bbE\Bigl(\bigl(\zeta^D_N(A^{(i)}\times(-\infty, M])\bigr)^2;\,h^{D_N}\leq m_N+u\Bigr)
		\quad
	\end{equation}
	and noting that the right-hand side of \eqref{E:4.40a} is additive in~$|A_N|$ we conclude
	\begin{equation}
		\limsup_{N\to\infty}\,\bbE\Bigl(\bigl(\zeta^D_N(D^\delta\times(-\infty, M])\bigr)^2;\,h^{D_N}\leq m_N+u\Bigr)
		\le c_\epsilon \rme^{(\alpha+\epsilon)u}
	\end{equation}
	for a constant~$c_\epsilon\in[1,\infty)$ depending on~$\ep$, $\delta$ and $D$.  Hence  we get~\eqref{eq:2m} with $\theta(u):=\inf_{0<\epsilon<1}[\epsilon u+\log c_\epsilon]$.  As this holds for all~$\epsilon\in(0,1)$, taking~$\epsilon\downarrow0$ along with~$u\to\infty$ yields~$\theta(u)=o(u)$ as~$u\to\infty$. The continuity follows from  the bounds $\theta(u)\le\theta(u+r)\le\theta(u)+r$ for all~$u,r\ge0$.
\end{proof}

The proof of Lemma~\ref{lem:2-density} constitutes the remainder of this subsection. Given~$x,y\in D_N$, let us abbreviate $\ell=\ell_N(x,y)$, set $L:= \lfloor \rme^\ell\rfloor$ and define
\begin{equation}
	\label{E:5.42}
		B_1 := B(x,L)  \cap \bbZ^2   \ , \qquad B_2 := B(y,L)  \cap \bbZ^2  \ , \qquad 
		W := B_1 \cup B_2 \ , \qquad B_3 := D_N \setminus (W\cup \partial_\infty W),
\end{equation}
where $\partial_\infty W$ denotes the set of all vertices whose $\ell^\infty$-distance to $W$ is $1$. Observe that $\partial W$ is a subset of $\partial B_3$ which is the set of all vertices whose $\ell^\infty$-distance to $x$ or $y$ is $L$. 
Also, pick $\ep'\in(0,1/3)$, let $W':=\{z\in D_N^\delta \colon \rmd_\infty(z,\partial_\infty W)>\ep'L \}$ and finally set
\begin{equation}
	\label{E:5.43}
	B_1' := B_1 \cap W' \ ,\quad B_2' := B_2 \cap W' \ ,\quad  B'_3 := B_3 \cap W' \,.
\end{equation}
 These definitions are also illustrated in Fig.~\ref{f:sec-moment}.

\begin{figure}
	\centering
	\includegraphics[width=0.55\textwidth]{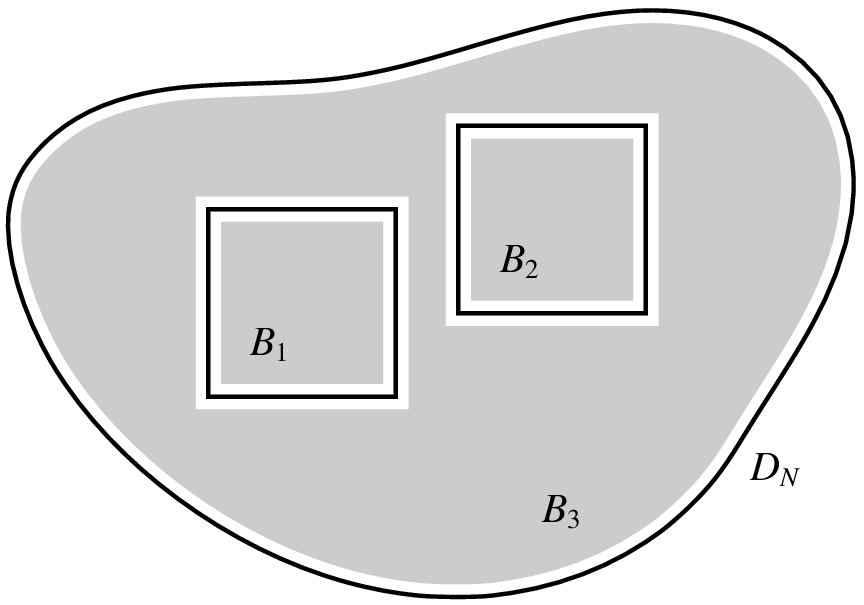}
	\caption{An illustration of the sets defined in \eqref{E:5.42} and \eqref{E:5.43}. The square boxes~$B_1$ and~$B_2$ are centered at vertices~$x$ and~$y$, respectively. The shaded area marks the set~$W'$ underlying the definitions of~$B_1'$, $B_2'$ and~$B_3'$.}
	\label{f:sec-moment}
\end{figure}

To bound $f_u^{N; x,y}(s,t)$, we will first condition on the values of $h^{D_N}$ on $\partial_\infty W $ and then use the Gibbs-Markov property, Lemma~\ref{prop:UB} and standard Gaussian estimates to bound this conditional density in terms of the unique bounded harmonic  extension~$\wt\varphi^N$  of $h^{D_N}|_{\partial_\infty W}$ onto all of $\bbZ^2$. Note that $\wt \varphi^N = \varphi^{D_N,D_N\smallsetminus \partial_\infty W}$ on $W\cup \partial_\infty W$. 

For our purposes, sufficient control over $\wt \varphi^N$  will be  achieved by using its value  
at $x$, together with its oscillation in~$W'$:
\begin{equation}
	R_N(x,y) :=\max_{z,z'\in W'}
	\bigl|\wt \varphi^N_z-\wt \varphi^N_{z'}\bigr|.
\end{equation}
The oscillation can be controlled  via:

\begin{lemma}
	\label{lem:sep-osc}
	Let $\delta,\ep'\in(0,1/3)$. There exists $C>0$ such that for all $N\ge1$, all $x,y\in D^\delta_N$, all $r\ge0$, and all $w\in\bbR$,
	\begin{equation}
		\label{E:4.39}
		\bbP\Bigl(R_N(x,y)\ge r\,\Big|\, \varphi_x^{D_N,D_N\setminus\partial_\infty W}=m_N-m_{\texte^{\ell}}+w\Bigr)
		\leq  C \exp\Bigl\{- C^{-1}[(r- C\tfrac{|w|}{n-\ell})^+]^2\Bigr\}\,,
		\quad
	\end{equation}
	where $\ell:=\ell_N(x,y)$.
\end{lemma}

\begin{proof}
	This is a reformulation of~\cite[Proposition B.5]{Ballot}. We take  $-x/N+D$  as~$U$ and $B(0,1)\cup B((y-x)/(L-1), 1)$ as~$W$ and use the discretization in~\cite{Ballot} for the inner domain.
\end{proof}

With this in hand, we are ready to give:

\begin{proof}[Proof of Lemma \ref{lem:2-density}]
	We may assume~$p$ to be so large such that $\ell:=\ell_N(x,y)$ obeys  $n-\ell\geq 1+\log \max\{c,\delta^{-1}\}$, where $c$ is the constant from Lemma~\ref{prop:UB} .
	Then,  in particular,  $x,y\in D^\delta_N$ implies
	$W \subseteq D^{\delta/2}_N$.  Given $v \in \bbR^{\partial_\infty W}$, let $f_u^{N; x,y}(s,t|v)$ be the  function   defined exactly as $f_u^{N; x,y}(s,t)$ in~\eqref{E:4.34} except that the probability is conditioned on  $\{{h^{D_N}}|_{\partial_\infty W} = m_N - m_{L} + v\}$.  Using the spatial Markov property of $h^{D_N}$, the measure $f_u^{N; x,y}(s,t|v) \rmd s \rmd t$ is now dominated by  the product of two conditional measures 
	\begin{gather}
		\label{e:5.47a}
		\bbP\Bigl(h^{B_1}_x - m_N \in\rmd s \,\Big|\, {h^{B_1}}|_{\partial B_1} = m_N - m_{L} + v \Bigr) \\ 
		\label{e:5.47b}
		\bbP\Bigl(h^{B_2}_y -  m_N\in \rmd t \,\Big|\, {h^{B_2}}|_{\partial   B_2  } = m_N - m_{L} + v \Bigr)
	\end{gather}
	and three conditional probabilities 
	\begin{gather}
		\label{e:5.47c}
		\bbP\Bigl(h^{B_3}\leq m_N + u \text{ on } B'_3\,\Big|\, {h^{B_3}}|_{\partial D_N} = 0 ,\,  {h^{B_3}}|_{\partial_\infty W} = m_N - m_{L} + v \Bigr)\\
		\label{e:5.47d}
		\bbP\Bigl( h^{B_1}\leq m_N + u \text{ on } B'_1\,\Big|\, {h^{B_1}}|_{\partial B_1} = m_N - m_{L} + v,\, h^{B_1}_x = m_N +s\Bigr)\\
		\label{e:5.47e}
		\bbP\Bigl( h^{B_2}\leq m_N + u \text{ on } B'_2\,\Big|\, {h^{B_2}}|_{\partial B_2} = m_N - m_{L} + v,\, h^{B_2}_y = m_N +t\Bigr) \,.
	\end{gather}
	We will now bound each of these terms separately to give an estimate on~$f_u^{N; x,y}(s,t|v)$ which can then be integrated into the desired bound.
	
	Let us write $\ol{v}$ for  the unique bounded harmonic extension of~$v$ from $\partial_\infty W$ to~$\bbZ^2$.  Abbreviate
	\begin{equation}
		w := \ol{v}_x\quad\text{and}\quad r := \max_{ z, z' \in W'} |\ol{v}_z -  \ol{v}_{z'}|.
	\end{equation}
	The  Gibbs-Markov decomposition then  dominates~\eqref{e:5.47a} by 
	\begin{equation}
		\BbbP\Bigl(h^{B_1}_x - m_{L} + \ol{v}_x \in \rmd s \Bigr) 
		\leq c \ell \rme^{-2\ell+\alpha(-s+w)} \rmd s\,,
	\end{equation}
	where we have used Lemma~\ref{lemma-3.2a} and~\eqref{item:ub-r:Gaussian} to bound the Gaussian density. Similarly, the second probability is at most
	\begin{equation}
		\BbbP\Bigl(h^{B_{ 2}}_x - m_{L} + \ol{v}_y \in \rmd t \Bigr)  \le  
		c \ell \rme^{-2\ell+\alpha(-t+\ol{v}_y)} \rmd t 
		\leq c \ell \rme^{-2\ell+\alpha(-t+w + r)} \rmd t,
	\end{equation}
	where the second inequality is based on $\ol{v}_y\le \ol{v}_x+r$. 
	
	For the probabilities  in \twoeqref{e:5.47c}{e:5.47e}  we use the upper  bounds  in Lemma~\ref{prop:UB}. By shifting $h^{B_3}$ by $-m_N - u$, the boundary conditions in~\eqref{e:5.47c} become
	$-m_N-u$ on $\partial D_N$ and $-m_{L} + v - u$ on $\partial_\infty W$, while the event whose probability we are after is now $\{{h^{B_3}}|_{B_3'} \leq 0\}$.
	As $N/L$  is sufficiently large  by our assumption on~$p$,  the first part of  Lemma~\ref{prop:UB}  with
	\begin{equation}
		-x/N+D\,,\quad B(0,1) \cup B((y-x)/ (L-1) ,1)\,,\quad N\quad\text{and}\quad L-1
	\end{equation}
	in place of
	$D$,~$\wt{D}$,~$N$  and  $K$,  respectively, 
	then dominates~\eqref{e:5.47c} by 
	\begin{equation}
		\begin{aligned}
		c \frac{(1+u)\bigl( ( w -r  - u )^- + \rme^{-(( w - 3r  - u)^+)^{3/2}}\bigr)}{n-\ell}
		&\leq 
		c' \frac{(1+u)\bigl( (w - u - r)^- + \rme^{-\rho(w - u -  3 r)^+}\bigr)}{n-\ell}
		\\& \leq
		c'' \frac{(1+u)
			(1+r) \rme^{ 3 \rho r} \bigl((w-u)^- + \rme^{-\rho(w - u)^+}\bigr)}{n-\ell} \,,
	\end{aligned}
	\end{equation}
	where we set $\rho := \alpha +2\ep$ and then choose $c', c'' < \infty$ accordingly. 
	
	Similarly, subtracting $m_N + u$ from  $h^{B_1}$ in~\eqref{e:5.47d} and  applying Lemma~\ref{prop:UB} with $D : =B(0,1)$ and $N:=L+1$ bounds the probability in~\eqref{e:5.47d} by
	\begin{equation}
		c \frac{\big(1+( w-3r-u)^-\big)\big(1+(s-u)^-\big)}{\ell}
		\leq 
		c' \frac{\big(1+u)^2(1+r)(1+w^-)(1+s^-)}{\ell} \,.
	\end{equation}
	An analogous bound  with~$s$ replaced by~$t$  applies to~\eqref{e:5.47e}. Combining  these observations,  $f_u^{N; x,y}(s,t|v)$ is at most a constant times
	\begin{equation}
		\label{e:5.56}
		\frac{(1+u)^5(1+s^-)(1+t^-)}{n-\ell}
		\rme^{-4\ell-\alpha (s+t)} 
		\\\times\bigg(
		\rme^{2\alpha w +  5\rho r}
		\bigl((w-u)^- + \rme^{-\rho(w - u)^+}\bigr)
		(1+r)^3(1+w^-)^2
		\bigg).
	\end{equation}
	Recalling the definition of $w$ and $r$, the integral with respect to $\bbP({h^{D_N}}|_{\partial_\infty W} - m_N + m_{L} \in \rmd v)$  of the expression inside the large brackets above is at most
	\begin{equation}
	\begin{aligned}
		\label{E:5.58a}
		\int_{w \in \bbR} \sum_{r \geq 1} \bbP\Bigl(\wt\varphi^N_x - m_N + m_{L} \in \rmd w)
		\bbP\Bigl(R_N(x,y)\ge r-1\,\Big|\,\wt\varphi^N_x - m_N + m_{L}= w\Bigr) \\
		\times\rme^{2\alpha w +  5 \rho r}
		\bigl((w-u)^- + \rme^{-\rho(w - u)^+}\bigr)
		(1+r)^3(1+w^-)^2 \,.
	\end{aligned}
	\end{equation}
	For the first probability  we first note that, by  Lemmas~\ref{lemma-3.2a} and~\ref{lem:Gibbs-Markov}, 
	\begin{equation}
		\label{e:varphi}
		\Var(\wt\varphi^N_x)
		 = \Var(h^{D_N}_x) - \Var(h^{B_1}_x)
		\le g\log (N/L) + O(1) \,.
	\end{equation}
	Abbreviating $m_{N,L}:= m_N - m_{N/L} - m_{L}$, Lemma~\ref{lem:Gaussian} shows 
	\begin{equation}
	\label{E:5.58b}
		\begin{aligned}
			\bbP\bigl(\wt\varphi^N_x- m_N + m_L\in \rmd w\bigr)
			= \bbP\bigl(\wt\varphi^N_x -m_{N/L} - m_{N,L} \in \rmd w\bigr)
			\le c N^{-2}\texte^{2\ell-\alpha (w+m_{N,L})}\log(N/L)\,\rmd w\\
			= c N^{-2}(\log N / \log L)^{3/2} (\log(N/L))^{-1/2} \texte^{2\ell-\alpha w } \rmd w,
		\end{aligned}
	\end{equation}
	where we used that $m_{N,L}=\tfrac34\sqrt{g}\bigl(-\log \log N + \log^+\log( N/ L) +\log^+\log^+ L\bigr)$ once $N,L\ge\texte$.  The second probability in \eqref{E:5.58a} is handled by Lemma~\ref{lem:sep-osc}.
	
 We now bound the integral $f^{N;x,y}_u(s,t)$ of $f_{u}^{N;x,y}(s,t|v)$ with respect to $\bbP({h^{D_N}}|_{\partial_\infty W} - m_N + m_{L} \in \rmd v)$ by combining~\eqref{e:5.56} with~\eqref{E:5.58a}, \eqref{E:5.58b} and Lemma~\ref{lem:sep-osc}  to get 
	\begin{equation}
		\begin{aligned}
			 \label{eq:f_{u}^{N;x,y}st-fin}
			&\frac{f_{u}^{N;x,y}(s,t)}
			{N^{-2} \rme^{-2\ell-\alpha (s+t)} (1+u)^5(1+s^-)(1+t^-)
				\big(\frac{n}{\ell(n-\ell)}\big)^{3/2}} \\
			& \quad \leq  c'
			\int \rmd w\; \rme^{\alpha w}
			\Big((w-u)^- +\rme^{-\rho(w-u)^+} \Big) (1+|w|)^2
			\sum_{ r=1}^\infty \rme^{-  6\rho\bigl( r- C\tfrac{|w|}{n-\ell}\bigr)}\rme^{  5\rho r} (1 + r)^3\\
			& \quad \leq c'' \int \rmd w\;\rme^{\alpha w  +\ep\abs{w}}
			\Big((w-u)^- + \rme^{-\rho(w-u)^+}\Big) \,,
		\end{aligned}
	\end{equation}
	for some constants $c',c''\in(0,\infty)$,  with  $C$ as in Lemma~\ref{lem:sep-osc} and  assuming that~$p$ is so large that  also  $6C(\alpha + 2\ep)/(n-\ell)<\ep/2$.  
	Splitting the integration domain into $w\in(u,\infty)$, $w\in[0,u]$, and $w\in(-\infty,0)$, the integral is at most a constant times $\rme^{ (\alpha + 2\epsilon) u}$. The claim  follows with $3 \epsilon$ replacing $\epsilon$ in the desired upper bound.
\end{proof}

We note that  the proof of Lemma~\ref{lem:2-density}  also relied on estimates from Lemma~\ref{lemma-3.2a} and is thus constrained to the  particular  discretization \eqref{E:2.1}.

\section{Proof of Theorem~\ref{thm-A}}
\label{sec5}\noindent
In this section we will give a formal proof of Theorem~\ref{thm-A}. For easier exposition, and also since this is the main novelty in the present work, we will first prove the convergence of the near-extremal process~$\zeta^D_N$ alone and deal with the joint convergence with the extremal process and the field itself later.

\subsection{Near-extremal process convergence}
Thanks to  Lemmas~\ref{lemma-tight} and~\ref{lemma-DGFF-tail}, the processes $\{\zeta^D_N\colon N\ge1\}$ are tight in the space of Radon measures on~$\overline D\times\overline\R$. (This uses that $\zeta^D_N$ is supported on~$D\times\R$ for each~$N\ge1$.) 
The strategy of the proof is similar to that used in \cite{BL1,BL2,BL3} for the extremal process and~\cite{BL4} for the intermediate DGFF level sets: We extract a subsequential limit~$\zeta^D$ of these processes and identify a properties that determine its law uniquely.

Our first item of concern is the dependence of the subsequential limit on the underlying domain~$D$. This is the content~of:

\begin{proposition}
	\label{lem:sim-lim}
	Every sequence of~$N$'s tending to infinity contains a subsequence $\{N_k\}_{k\ge1}$ and, for all~$D\in\mathfrak D$, there is a random measure~$\zeta^D$ such that
	\begin{equation}
		\label{e:3.1} 
		\zeta^D_{N_k} \,\,\,\underset{k \to \infty} \Lawarrow \,\,\,\zeta^D,\quad D\in\mathfrak D.
	\end{equation}
	The same applies (with the same subsequence and the limit measure) even if~$\zeta^D_N$ is defined using a sequence of  lattice  domains~$\{D_N\}_{N\ge1}$ satisfying \twoeqref{E:2.1a}{E:2.2a}.
\end{proposition}

The proof is based on two lemmas that will be of independent interest.
We start by recording a consequence of the first moment estimates from Section~\ref{sec4}:

\begin{lemma}[Stochastic absolute continuity]
	\label{lem:stoc_cont}
	Let $D\in\mathfrak D$ and let~$\zeta^D$ be a subsequential weak limit of the measures $\{\zeta^D_N\colon N\ge1\}$. Then for all measurable  $A \subseteq \overline D$, 
	\begin{equation}
		\label{E:5.2}
		\leb(A)= 0\quad\Rightarrow\quad\zeta^D(A\times\overline\R) = 0\text{\rm\ \ a.s.}
	\end{equation}
	The measure~$\zeta^D$ is concentrated on~$D\times[0,\infty)$ with $\zeta^D(D\times[0,\infty))<\infty$ a.s.
\end{lemma}

\begin{proof}
	Let~$A \subseteq \overline D$ be measurable  with~$\Leb(A)=0$.  The outer regularity of the Lebesgue measure implies the existence of functions $\{f_j\colon j\ge1\}\subseteq C_\cc(\R^2)$  with support restricted to an open ball containing~$\overline D$  such that $\Ind_A\le f_j$ and $\Vert f_j\Vert\le1$ hold for all~$j\ge1$ and~$f_j\downarrow \Ind_A$ Lebesgue-a.e.\ as~$j\to\infty$. Given~$\epsilon>0$, we can thus estimate
	\begin{equation}
		\label{E:5.3oi}
		\BbbP\bigl(\zeta^D(A\times\overline\R)\ge\epsilon\bigr)\le\inf_{j\ge1}\BbbP\bigl(\langle\zeta^D,f_{j}  \otimes \Ind_{\overline\R}  \rangle\ge\epsilon\bigr)
		\le\inf_{j\ge1}\,\limsup_{k\to\infty}\BbbP\bigl(\langle\zeta^D_{N_k},f_{j}  \otimes \Ind_{\R}  \rangle\ge\epsilon\bigr),
	\end{equation}
	where~$\{N_k\}_{k\ge1}$ is a subsequence such that~$\zeta_{N_k}^D\Lawarrow\zeta^D$.
	
	Pick~$u>0$.  Thanks to the Markov inequality, the probability on the right hand side of~\eqref{E:5.3oi} is at most
	\begin{equation}
		\label{eq:La-unif-Markov}
		\epsilon^{-1}
		\bbE\Bigl(\langle\zeta^D_{N},f_{j}  \otimes \Ind_{\R}  \rangle\,;\,h^{D_N}\leq m_N+u\Bigr)
		+\bbP\Bigl(\,\max_{x\in D_N} h^{D_N}_x> m_N+u\Bigr)
	\end{equation}
	with~$N:=N_k$.  Lemma~\ref{lemma-tight} along with a routine approximation   argument   give 
	\begin{equation}
		\limsup_{N\to\infty}\bbE\Bigl(\langle\zeta^D_{N},f_{j} \otimes \Ind_{\R} \rangle\,;\,h^{D_N}\leq m_N+u\Bigr)
		\le c u\int_D f_j(x)\textd x
	\end{equation}
	and so, by the  Bounded Convergence Theorem, the expectation in \eqref{eq:La-unif-Markov} tends to zero as~$N\to\infty$ followed by~$j\to\infty$. Lemma~\ref{lemma-DGFF-tail} in turn shows that the probability on the right of \eqref{eq:La-unif-Markov} vanishes as~$N\to\infty$ followed by~$u\to\infty$ and so we~get~\eqref{E:5.2}.
	
	The same argument applied to~$A:=\partial D$ shows that~$\zeta^D$ puts no mass on~$\partial D\times\overline\R$ a.s. The bound~\eqref{e:5.24} shows that~$\zeta^D$ does not charge $D\times\{+\infty\}$. By Lemma~\ref{lemma-DGFF-tail}, the maximum of~$h^{D_N}$ is at most $m_N+o(\sqrt{\log N})$ and so $\zeta^D(D\times[-\infty,0))=0$ a.s. The a.s.-finiteness of $\zeta^D(D\times[0,\infty)  )  $ then follows from Lemma~\ref{lem:tightness}.
\end{proof} 

Next we will state the consequence Gibbs-Markov property: 

\begin{lemma}[Gibbs-Markov property]
	\label{lem:GibbsMarkov}
	Let $D,\wt{D} \in\mathfrak D$ obey $\wt{D} \subseteq D$ and $\leb(D\smallsetminus\wt D)=0$. Suppose that $\zeta^D$ and $\zeta^{\wt{D}}$ are subsequential weak limits of $\{\zeta^D_{N}\colon N\ge1\}$ and $\{\zeta^{\wt{D}}_{N}\colon N\ge1\}$ respectively, along the same subsequence. Then
	\begin{equation}
		\label{E:5.6i}
		\zeta^{D}(\rmd x \,\rmd t) \laweq
		\rme^{\alpha\Phi^{D,\wt{D}}(x)} \zeta^{\wt{D}}(\rmd x \,\rmd t),
	\end{equation}
	where $\Phi^{D,\wt{D}}=\NN(0,C^{D,\wt D})$ is independent of $\zeta^{\wt{D}}$ on the right-hand side.
\end{lemma}

\begin{proof} 
	We remark that the fact that~$\leb(D\smallsetminus\wt D)=0$ ensures, via Lemma~\ref{lem:stoc_cont}, that neither~$\zeta^D$ nor~$\zeta^{\wt D}$ charge the set~$(D\smallsetminus\wt D)\times\R$ a.s.\ and so it is immaterial that~$\Phi^{D,\wt D}$ is not defined there. 
	
	Pick $f\in C_\cc(\wt D\times\R)$ with~$f\ge0$. Noting that~$\wt D_N\subseteq D_N$, let $h^{\wt D_N}$, resp.,~$\varphi^{ D_N ,\wt D_N}$ be independent fields with the law of DGFF in~$\wt D_N$, resp., the ``binding'' field between domains~$D_N$ and~$\wt D_N$. Lemma~\ref{lem:Gibbs-Markov} then shows that~$h^{D_N}:= h^{\wt D_N}+\varphi^{D_N,\wt D_N}$ has the law of DGFF in~$D_N$. 
	Define the functions
	\begin{equation}
		e_N(x):=\exp\big\{\alpha \varphi^{D_N,\wt D_N}(\lfloor N x \rfloor)\big\}
	\end{equation}
	and
	\begin{equation}
		f_N(x,t) := f \Bigl(x, \, t - \varphi^{D_N,\wt D_N}(\lfloor Nx \rfloor) /\sqrt{\log N} \Bigr).
	\end{equation}
	Using~$h^{D_N}$ to define~$\zeta^D_N$ and~$h^{\wt D_N}$ to define~$\zeta^{\wt D}_N$, we then have
	\begin{equation}
		\label{e:3.13}
		\la \zeta^D_N ,\, f \ra = \la \zeta^{\wt D}_N ,\, e_N f_N \ra \,,
	\end{equation}
	where~$\zeta^{\wt D}_N$ is independent of~$\varphi^{D_N,\wt D_N}$ (implicitly contained  in~$e_N$ and  $f_N$) on the right-hand side. Hereby we get
	\begin{equation}
		\label{e:3.15a}
		\Bigl|\la \zeta^D_N ,\, f \ra-\bigl\la \zeta^{\wt D}_N ,\, \texte^{\alpha\Phi^{D,\wt D}}f \bigr\ra\Bigr|\le
		\zeta^{\wt D}_N(D\times\R)\bigl\Vert e_Nf_N-\texte^{\alpha\Phi^{D,\wt D}}f\bigr\Vert_\infty.
	\end{equation}
	The first term on the right-hand side is tight thanks to Lemma~\ref{lem:tightness}. The uniform continuity of~$f$ and the coupling in Lemma~\ref{lemma-3.5} show that the second term tends to~$0$ in probability as $N \to \infty$. Since~$(x,t)\mapsto\texte^{\alpha\Phi^{D,\wt D}(x)}f(x,t)$ is uniformly continuous a.s., if~$\{N_k\}_{k\ge1}$ is a sequence along  which~$\zeta^{\wt D}_{N_k}$ tends to~$\zeta^{\wt D}$, by conditioning on~$\Phi^{D,\wt D}(x)$ we get
	\begin{equation}
		\label{E:5.11}
		\E\bigl(\texte^{-\la \zeta^{\wt D}_{N_k} ,\, \texte^{\alpha\Phi^{D,\wt D}}f \ra}\bigr)\,\,\underset{k\to\infty}\longrightarrow\,\,\E\bigl(\texte^{-\la \zeta^{\wt D} ,\, \texte^{\alpha\Phi^{D,\wt D}}f \ra}\bigr).
	\end{equation}
	Applying the subsequential convergence to the first term on the left of \eqref{e:3.15a} as well, we infer
	\begin{equation}
		\E\bigl(\texte^{-\la \zeta^{\wt D} ,\, \texte^{\alpha\Phi^{D,\wt D}}f \ra}\bigr)=\E\bigl(\texte^{-\la \zeta^{ D} ,\, f \ra}\bigr),\quad f\in C_\cc(\wt D\times\R).
	\end{equation}
	The equality extends to all~$f\in C(\overline D\times\overline\R)$ by approximating~$f$ on~$\wt D\times\R$ by an increasing sequence of functions in $C_\cc(\wt D\times\R)$ and invoking the Bounded Convergence Theorem along with Lemma~\ref{lem:stoc_cont}. Hereby we get \eqref{E:5.6i}.
\end{proof} 

We are ready to give:

\begin{proof}[Proof of Proposition~\ref{lem:sim-lim}]
	Let $\fU_0$ be the collection of domains in~$\R^2$ consisting of finite unions of non-degenerate rectangles with  corners  in $\bbQ^2$. Since $\fU_0$ is countable, given a sequence of~$N$'s tending to infinity, Cantor's diagonal argument permits us to extract a strictly increasing subsequence $\{N_k\}_{k\ge1}$ such that \eqref{e:3.1} holds for all~$D\in\mathfrak D_0$. We thus have to show that the convergence extends (with the same subsequence) to all~$D\in\mathfrak D$.
	
	Let~$D\in\mathfrak D$ and, for each~$n\ge1$, let $D^n$ be the union of all open rectangles contained in~$D$ with  corners  in $(2^{-n}\Z)^2$. Since~$D$ is open, we have~$D^n\ne\emptyset$ once~$n$ is sufficiently large and $D^n\uparrow D$ as~$n\to\infty$.
	Moreover, if~$f\in C_\cc(D\times\R)$, then $\supp(f)\subseteq \overline{D^n}\times[-n,n]$ for all~$n$ sufficiently large, say~$n\ge n_f$.
	Hence, for every limit $\zeta^D$ of $(\zeta^D_{N_k})$ along a subsubsequence $(k_\ell)$ and every $f\in C_\cc(D\times\R)$,
	\begin{equation}
		\la \zeta^D, f \ra
		= \lim_{n\to\infty}\la \zeta^{D^n}, \texte^{\Phi^{D,D^n}}f \ra
	\end{equation}
	by Lemma~\ref{lem:Gibbs-Markov}.
	As the right-hand side and hence $\zeta^D$ does not depend on the subsubsequence $(k_\ell)$, it follows that $\zeta^D_{N_k}$ converges  in law  to $\zeta^D$.
\end{proof}

In order to show that the limit measure does not depend on the underlying sequence, we need to verify the conditions of Theorem~\ref{BL2-thm2.8}. Most of these conditions have already been checked; one that still requires some work is the content of:

\begin{lemma}
	\label{lem:La-unif}
	Let   $\delta\in(0,1)$ and let $T,T_\delta\in\mathfrak{D}$ be equilateral triangles centered  at the origin of~$\R^2$  with side lengths $1$ and $1-\delta$, respectively,  such that~$T_\delta$ is a dilation of~$T$.  For each~$K\ge1$ natural, let~$\zeta^{K^{-1}T}$ be a subsequential weak limit of~$\{\zeta^{K^{-1}T}_N\colon N\ge1\}$, using the same subsequence as ensured by Proposition~\ref{lem:sim-lim}. Then for any $f\in C_\cc(\R)$ satisfying~$f\ge0$, 
	\begin{equation}
		\label{eq:La-unif}
		\lim_{\lambda\downarrow 0}\sup_{K\geq 1}\Biggl|\frac{\bbE\left(K^4
			\big\la \zeta^{K^{-1}T},\Ind_{K^{-1}T_\delta }\otimes f \big \ra
			\,\texte^{-\lambda K^4\la \zeta^{K^{-1}T},\Ind_{K^{-1}T}\otimes f  \ra}\right)}{\log(\ffrac1\lambda)}
		-c_\star\int_0^\infty f(s)\;s\;\rme^{-\frac{s^2}{2g}}\;\rmd s
		\int_{T_\delta} r^{T}(x)^2\;\rmd x\Biggr|=0,
	\end{equation}
	where~$c_\star$ is the constant from Theorem~\ref{thm-A}.
\end{lemma} 

\begin{proof}
	Let~$\tilde f=f_1\otimes f_2$ for
	$f_1 \in C_\cc(\overline { T })$ and $f_2 \in C_\cc(\R)$,
	and let~$\{N_j\}_{j\ge1}$ be the sequence defining the limiting measures $\{\zeta^{K^{-1} T }\colon K\ge1\}$ in the statement. The proof hinges on the convergence statement
	\begin{equation}
		\label{E:5.18}
		\bigl\langle\zeta_{\lfloor N_j/K\rfloor}^{ T },\tilde f\bigr\rangle\,\,\underset{j\to\infty}\Lawarrow\,\,
		K^4\bigl\langle\zeta^{K^{-1} T },\tilde f_K\bigr\rangle,
	\end{equation}
	where we set
	\begin{equation}
		\tilde f_K(x,t):= \tilde f(xK,t).
	\end{equation}
	Indeed, noting that $(\Ind_{ T_\delta}\otimes f)_K=\Ind_{K^{-1}  T_\delta }\otimes f$, we can rewrite the quantity in \eqref{eq:La-unif} with $\Ind_{ T_\delta }\otimes f$ replaced by~$\tilde f$ (while leaving the function in the exponent unchanged) as
	\begin{equation}
		\lim_{\lambda\downarrow 0}\sup_{K\geq 1}\lim_{j\to\infty}\Bigg|\frac{\bbE\left(\langle\zeta^{ T}_{\lfloor N_j/K\rfloor},\tilde f\bigr\rangle
			\,\texte^{-\lambda \langle\zeta^{ T }_{\lfloor N_j/K\rfloor},\Ind_{ T  }\otimes f\rangle}\right)}{\log(\ffrac1\lambda)}
		-c_\star\int_{ T\times[0,\infty)}\;s\;\rme^{-\frac{s^2}{2g}}
		r^{ T }(x)^2\;\tilde f(x,s)\;\rmd x\rmd s\Bigg|,
	\end{equation}
	The $j\to\infty$ limit can be bounded by the \emph{limes superior} as~$N\to\infty$ of the same expression with $\lfloor N_j/K\rfloor$ replaced by~$N$. This \emph{limes superior} no longer depends on~$K$, rendering the supremum over~$K\ge1$ redundant. Moreover, the $N\to\infty$ limit vanishes by Lemma~\ref{lem:tilted}.
	 Approximating $\Ind_{T_\delta}$ by functions in $C_\cc(T)$ then extends the limit to $\tilde f$ replaced by $\Ind_{T_\delta}\otimes f$. 
	
	Moving to the proof of \eqref{E:5.18}, we now let~$\tilde f\in C_\cc( T\times\R)$. Then there exists an open set~$B$ with~$\overline B\subseteq  T$ such that~$\supp( \tilde f )\subseteq B\times\R$. For all naturals~$K\ge1$ we then have
	\begin{equation}
		\label{E:5.20}
		\{x\in  T_N: xK/N \in B\}\subseteq  T_{\lfloor N/K\rfloor}\subseteq (K^{-1} T)_N
	\end{equation}
	once~$N$ is sufficiently large.
	Using the Gibbs-Markov property to write $h^{(K^{-1} T)_N}= h^{ T_{\lfloor N/K\rfloor}}+\varphi^{(K^{-1} T)_N, T_{\lfloor N/K\rfloor}}$ and denoting
	\begin{equation}
	\begin{aligned}
		\tilde f_{N,K}(&x,t):=K^4\texte^{\alpha(m_{\lfloor N/K\rfloor}-m_N)}\frac{\log {\lfloor N/K\rfloor}}{\log N} \\
		&\times
		\texte^{\alpha\varphi^{(K^{-1} T)_N, T_{\lfloor N/K\rfloor}}_x}
		\tilde f\biggl(x\frac{K\lfloor N/K\rfloor}N, t\sqrt{\frac{\log\lfloor N/K\rfloor}{\log N}}+\frac{m_N-m_{\lfloor N/K\rfloor}}{\sqrt{\log N}}-\frac{\varphi_x^{(K^{-1} T)_N, T_{\lfloor N/K\rfloor}}}{\sqrt{\log N}}\biggr),
	\end{aligned}
	\end{equation}
	which depends on the sample of $\varphi^{(K^{-1} T)_N, T_{\lfloor N/K\rfloor}}$,
	then gives
	\begin{equation}
		\label{E:NK-N}
		\begin{aligned}
			K^4\bigl\langle&\zeta^{K^{-1} T}_N,\tilde f_K\bigr\rangle
			=\frac{K^4}{\log N}\sum_{x\in (K^{-1} T)_N}\texte^{\alpha(h_x^{(K^{-1} T)_N}-m_N)}\,
			\tilde f\Bigl(x\frac K{N},\frac{m_N-h_x^{(K^{-1} T)_N}}{\sqrt{\log N}}\Bigr)
			\\
			&=\frac1{\log {\lfloor N/K\rfloor}}\sum_{x\in { T}_{\lfloor N/K\rfloor}}\texte^{\alpha(h_x^{ T_{\lfloor N/K\rfloor}}-m_{\lfloor N/K\rfloor})}\,
			\tilde f_{N,K}\Bigl(\frac x{\lfloor N/K\rfloor},\frac{m_{\lfloor N/K\rfloor}-h_x^{ T_{\lfloor N/K\rfloor}}}{\sqrt{\log \lfloor N/K\rfloor}}\Bigr)
			\\
			&=\bigl\langle\zeta_{\lfloor N/K\rfloor}^{ T},\tilde f_{N,K}\bigr\rangle
		\end{aligned}
	\end{equation}
	provided~$N$ is so large that the left inclusion in \eqref{E:5.20} applies.
	 By Lemma~\ref{l:cd}, we have
	\begin{equation}
		\max_{\substack{x\in  T_N\colon\\ xK/N\in B}}\bigl|\varphi^{(K^{-1} T)_N, T_{\lfloor N/K\rfloor}}_x\bigr|\,\underset{N\to\infty}\longrightarrow\,0
	\end{equation}
	in probability.
	 Since also  $\alpha(m_N-m_{\lfloor N/K\rfloor})\to 4\log K$ as~$N\to\infty$, we now get  from uniform continuity of $\tilde f$ in the second component  that $\Vert \tilde f_{N,K}- \tilde f\Vert_\infty$
	tends to zero in probability as~$N\to\infty$. As $\zeta_{\lfloor N/K\rfloor}^{ T}$ is tight, also
	$\bigl\langle\zeta_{\lfloor N/K\rfloor}^{ T},\tilde f_{N,K}\bigr\rangle - \bigl\langle\zeta_{\lfloor N/K\rfloor}^{ T},\tilde f\bigr\rangle$
	converges to zero in probability as $N\to\infty$.
	The left-hand side of~\eqref{E:NK-N} converges to $K^4\langle\zeta^{K^{-1}T},\tilde f_K\rangle$ by choice of $\{N_j\}_{j\ge 1}$. We can now conclude using Slutzky's lemma that \eqref{E:5.18} follows as desired.
\end{proof}

We are now ready for stating and proving the part of Theorem~\ref{thm-A} dealing solely with the near-extremal process:

\begin{theorem}[Near-extremal process convergence]
	\label{thm-5.5}
	In the limit~$N\to\infty$, the measures $\{\zeta^D_N\colon N\ge1\}$ converge in law to the measure~$\zeta^D$ in \eqref{E:2.3}.
\end{theorem}

\begin{proof}
	For each~$D\in\mathfrak D$, let~$\zeta^D$ be a weak subsequential limit of $\{\zeta^D_N\colon N\ge1\}$ along the same subsequence as ensured by Proposition~\ref{lem:sim-lim}. Pick~$f\in C_\cc(\R)$ non-negative with $\leb(f>0)>0$. Define the Borel measure
	\begin{equation}
		M^{D}(B):=\big\la\zeta^{D},\Ind_B\otimes f\big\ra.
	\end{equation}
	We will now check that~$M^D$ satisfies the assumptions of Theorem~\ref{BL2-thm2.8}. That~$M^D$ is concentrated on~$D$, of finite total mass and not charging Lebesgue null sets a.s.\ was shown in Lemma~\ref{lem:stoc_cont} (using that~$f$ is bounded). The factorization \eqref{E:2.18ua} over disjoint domains  follows from the fact that if~$D\cap\wt D=\emptyset$ then  the restrictions of $h^{D_N\cup \wt D_N}$ to $D_N$ and $\wt D_N$ are independent of each other.  The Gibbs-Markov property \eqref{E:2.19ua} was shown in Lemma~\ref{lem:GibbsMarkov}.  The uniform Laplace transform tail \eqref{E:1.25a} was verified in Lemma~\ref{lem:La-unif} with constant~$c$ given by
	\begin{equation}
		c:=c_\star\int_0^\infty f(t) t\texte^{-\frac{t^2}{2g}}\textd t.
	\end{equation}
	Invoking \eqref{E:3.25i} in Theorem~\ref{BL2-thm2.8} we conclude that, for each~$f\in C_\cc(\R)$ non-negative with non-zero total integral, the measure
	\begin{equation}
		\label{E:5.25}
		W_f^D(B):=\Bigl(c_\star\int_0^\infty f(t) t\texte^{-\frac{t^2}{2g}}\textd t\Bigr)^{-1}\big\la\zeta^{D},\Ind_B\otimes f\big\ra.
	\end{equation}
	has the law of cLQG. 
	
	Moreover, since the Gibbs-Markov property \eqref{E:2.19ua} uses the same ``binding'' field for all~$f$ as above, Theorem~\ref{thm-3.10} ensures that, for all $f_1,f_2\in C_\cc(\R)$ non-negative and not vanishing Lebesgue-a.e.,
	\begin{equation}
		W_{f_1}^D = W_{f_2}^D \quad\text{a.s.}
	\end{equation}
	The separability of~$C_\cc(\R)$ permits us to choose the null event so that the equality holds for all~$f_1,f_2\in C_\cc(\R)$ simultaneously a.s. Standard extension arguments then show that~$\zeta^D$ takes the form \eqref{E:2.3}.
\end{proof}

\subsection{Adding the extremal process}
Our next item of business is  adding  the convergence of the extremal process~$\eta^D_N$  to that of $\zeta^D_N$.  We start by recalling from~\cite{BL3} the structured version of the extremal process.
We denote $\Lambda_r(x):=\{y\in\Z^2\colon \rmd_\infty(x,y)<r\}$. Given any positive sequence $\{r_N\}_{N\ge1}$, define a measure on $D\times\R\times\R^{\Z^2}$ by
\begin{equation}
	\tilde\eta^D_N:=\sum_{x\in D_N}\Ind_{\{h^{D_N}_x=\max_{z\in\Lambda_{r_N}(x)}h^{D_N}_z\}}
	\,\delta_{x/N}\otimes\delta_{h^{D_N}_x-m_N}\otimes\delta_{\{h^{D_N}_x-h^{D_N}_{x+z}\colon z\in\Z^2\}}.
\end{equation}
This measure records the positions, centered values and configuration  near and  relative to the $r_N$-local maxima of~$h^{D_N}$. Our aim is to show:

\begin{theorem}[Joint convergence of near-extremal and extremal process]
	\label{thm-5.6}
	For any $D\in\mathfrak D$ and any $\{r_N\}_{N\ge1}$ with $r_N\to\infty$ and~$N/r_N\to\infty$,
	\begin{equation}
		\bigl(\zeta^D_N,\tilde\eta^D_N\bigr)\,\,\underset{N \to \infty}{\Lawarrow}\,\, (\zeta^D,\tilde\eta^D),
	\end{equation}
	where~$\zeta^D$ is derived from a sample ~$Z^D$ of cLQG in~$D$ as in \eqref{E:2.3} and, conditional on~$Z^D$,
	\begin{equation}
		\label{e:etaPPP}
		\tilde\eta^D=\text{\rm PPP}\Bigl(\bar c\,Z^D(\textd x)\otimes\texte^{-\alpha h}\textd h\otimes\nu(\textd\phi)\Bigr)
	\end{equation}
	where~$\bar c$ and $\nu$ are as in \eqref{E:1.3a}. More precisely, the joint law of $(\zeta^D,\tilde\eta^D)$ is such that, for each non-negative $f\in C_\cc(D\times\R)$ and $\tilde f\in C_\cc(D\times\R\times\R^{\Z^2})$,
	\begin{equation}
	\begin{aligned}
		\label{E:5.30}
		E\bigl(\texte^{-\langle\zeta^D,f\rangle-\langle\tilde\eta^D,\tilde f\rangle}\bigr)
		=E\biggl(\exp\Bigl\{-c_\star\int_D &Z^D(\textd x)\int_0^\infty\textd t\, t\texte^{-\frac{t^2}{2g}}\,f(x,t)
		\\
		&-\bar c\int_D Z^D(\textd x)\int_{\R}\textd h\,\texte^{-\alpha h}\int_{\R^{\Z^2}}\nu(\textd\phi)\bigl(1-\texte^{-\tilde f(x,h,\phi)}\bigr)\Bigr\}\biggr),
	\end{aligned}
	\end{equation}
	where the expectation  on the right  is over the law of~$Z^D$.
\end{theorem}

 From the family of measure-valued pairs  $\{(\zeta^D_N,\tilde\eta^D_N)\colon N\ge1\}$,  we may extract  a (joint) subsequential  weak  limit  (based on vague topology) to be called, with some abuse of notation, $(\zeta^D,\tilde\eta^D)$ for the remainder of this subsection.
The marginal distributions are known from Theorem~2.1 of~\cite{BL3} and from Theorem~\ref{thm-5.5}.
The  main task  is to characterize the joint law of this pair uniquely.

One way to achieve this is to follow the proofs of~\cite{BL1,BL2,BL3} while keeping track of the behavior of the near-extremal limit process throughout the manipulations. We will instead proceed by a direct argument that consists of two steps: We extract the spatial part of the intensity measure of~$\tilde\eta^D$ by way of a suitable limit and then equate it with the spatial part of the measure~$\zeta^D$. The first step is the content of:

\begin{lemma}
	\label{lemma-lln}
	For all~$D\in\mathfrak D$, the vague limit in probability
	\begin{equation}
		\label{E:6.30i}
		\vartheta^D(\cdot):=\lim_{s\to\infty}\,\,\alpha\,\rme^{-\alpha s}
		\tilde\eta^D\bigl(\cdot \times (-s,\,\infty)\times\R^{\Z^2}\bigr) 
	\end{equation}
	exists and is equal in law to $\bar c Z^D$.
	Moreover, the conditional distribution of $\tilde\eta^D$ given $\vartheta^D$ is that of a Poisson point process with intensity 
	$\vartheta^D(\textd x)\otimes\texte^{-\alpha h}\textd h\otimes\nu(\textd\phi)$ and, for any $f\in C(D)$, 
	\begin{equation}
		\label{E:6.30f}
		\lim_{s\to\infty}\,\,\alpha\,\rme^{-\alpha s}
		\tilde\eta^D\bigl(\textd x \times (-s-f(x),\,\infty)\times\R^{\Z^2}\bigr)
		=\texte^{\alpha f(x)}\vartheta^D(\textd x) 
	\end{equation}
	vaguely in probability.
\end{lemma}

\begin{proof}
	By~\cite[Theorem 2.1]{BL3}, $\tilde\eta^D$ is the Cox process $\text{\rm PPP}(W^D(\textd x)\otimes\texte^{-\alpha h}\textd h\otimes\nu(\textd\phi))$, where~$W^D$ is a random Radon measure on~$\overline D$ with $W^D\laweq \bar c\,Z^D$. Conditionally on $W^D$,  for each $s>0$ and each $B\subseteq D$ measurable, 
	\begin{equation}
		\tilde \eta^D\biggl(\Bigl\{(x,h,\phi)\in B\times\R\times\R^{\Z^2}\colon h\in\bigl(-s -f(x),\,\infty\bigr)\Bigr\}\biggr)
		\laweq \text{Poisson}\Bigl(\int_B W^D(\textd x)\alpha^{-1}\texte^{\alpha (s+f(x))}\Bigr)  .
	\end{equation}
	 From the fact that $\lambda^{-1}\text{Poisson}(a\lambda)\to a$ in probability as~$\lambda\to\infty$, applied to $\lambda:=\alpha^{-1}\rme^{\alpha s}$ and $a:=\int_B W^D(\rmd x)\rme^{\alpha f(x)}$,
	we obtain~\eqref{E:6.30f} and (taking $f=0$) also~\eqref{E:6.30i}.
	Similarly, for a collection of disjoint sets $B_i\subseteq D$, $i=1,\ldots,k$ with $\Leb(\partial B_i)=0$, we obtain
	\begin{equation}
		\big(\vartheta^D(B_i) \big)_i =  \Big(\lim_{s\to\infty} \alpha \rme^{-\alpha s} \tilde\eta^D(B_i\times(-s,\infty)\times\bbR^{\bbZ^2} \big)\Big)_i
		=\big( W^D(B_i) \big)_i\quad\text{a.s.},
	\end{equation}
	where we used~\eqref{E:6.30i} and Lemma~\ref{lem:stoc_cont} for the first equality, and argued as above for the second equality.
	It follows that  the conditional distribution of $\tilde\eta^D$ with respect to $\vartheta^D$ is a.\,s.\ equal to that with respect to $W^D$.
\end{proof}

Let $\tilde\zeta^D:=\zeta^D(\cdot\times\R)$.
In order to identify the joint law of $(\tilde\zeta^D,\vartheta^D)$,  where $\vartheta^D$ is as in \eqref{E:6.30i}, we will rely on Theorem~\ref{thm-3.10} for which  we need:

\begin{lemma}[Joint Gibbs-Markov property]
	\label{lemma-GM2}
	For any $D,\wt{D} \in\mathfrak D$ satisfying $\wt{D} \subseteq D$ and $\leb(D\smallsetminus\wt D)=0$,  and any subsequential limits along the same sequence of~$N$'s, 
	\begin{equation}
		\Bigl(\tilde\zeta^{D}(\rmd x),\vartheta^D(\textd  \tilde x)\Bigr) \laweq
		\Bigl(\rme^{\alpha\Phi^{D,\wt{D}}(x)} \tilde\zeta^{\wt{D}}(\rmd x ),\rme^{\alpha\Phi^{D,\wt{D}}( \tilde x)} \vartheta^{\wt{D}}(\rmd  \tilde x)\Bigr),
	\end{equation}
	where $\Phi^{D,\wt{D}}=\NN(0,C^{D,\wt D})$ is independent of $(\tilde\zeta^{\wt{D}},\vartheta^{\wt D})$ on the right-hand side.
\end{lemma}

\begin{proof}
	By the Gibbs-Markov property  (Lemma~\ref{lem:Gibbs-Markov})  for the DGFF $h^{D_N}$, 
	there is a coupling of $h^{D_N}$ and $h^{\wt D_N}$ such that $h^{D_N}=\varphi^{D_N,\wt D_N} +  h^{\wt D_N}$, with  the fields on the right-hand side being independent.
	Let $\tilde\zeta^D_N$, $\tilde\eta^D_N$ be derived from $h^{D_N}$, and $\tilde\zeta^{\wt D}_N$, $\tilde\eta^{\wt D}_N$ from $h^{\wt D_N}$,
	using  the same $\{r_N\}_{N\ge 1}$ with $r_N\to\infty$ and $N/r_N\to\infty$.
	We abbreviate $\varphi^N:=\varphi^{D_N,\wt D_N}$.
	
	For any $f\in C_\cc(\wt D)$, we have
	\begin{equation}
		\label{e:zetaN-GM}
		\la \tilde \zeta^D_N,\, f\ra
		=\la \tilde \zeta^{\wt D}_N,\, \rme^{\alpha\varphi^N}f\ra
	\end{equation}	
	where we used that~$f$ does not depend on the field coordinate and so is unaffected by the shift by~$\varphi^N$. 
	For the extremal process, we will focus on quantities of the form 
	\begin{equation}
		\label{E:6.34i}
		\la \tilde \eta^D_N,\, f \otimes \Ind_{(-s,\,\infty)}\otimes 1 \ra
	\end{equation}
	where $s>0$.  
	Our aim is to relate \eqref{E:6.34i} to
	\begin{equation}
		\bigl\la \tilde \eta^{\wt D}_N,\, f \otimes \Ind_{(-s -\Phi^{D,\wt D},\,\infty)}\otimes 1\bigr\ra 
	\end{equation}
	where~$\Phi^{D,\wt D}$ is the continuum  ``binding''  field  (regarded as independent of~$\eta^{\wt D}_N$). Besides the fact that the test function is not continuous in the field variable, here we are faced with the additional difficulty that the additive term~$\varphi^N$ may change the spatial positions of relevant $r_N$-local extrema.

	Let $a>0$ be such that $\supp(f)\subset \wt D^{a}:= \{x\in \wt D\colon\rmd_\infty(x,\wt D^\cc)>a\}$.
	To control the fluctuation of the ``binding'' field at scale $2r_N$, we define
	\begin{equation}
		L_N:=\max_{\substack{x,y\in \wt D^a_N\\\rmd_\infty(x,y)\leq 2r_N}}
		\big|\varphi^N_x-\varphi^N_y\big|\,.
	\end{equation}
	Given  $\tilde s>0$,  for each~$N\ge1$ and each $\delta>0$ consider the event 
	\begin{equation}
		F_N:=\bigcap_{x\in \wt D^a_N} \biggl\{h_x^{\wt D_N}\ge m_N-\wt s\,\text{ implies }\,\max_{y\in\Lambda_{2r_N}(x)\smallsetminus\Lambda_{r_N/2}(x)}h^{\wt D_N}_y\le h^{\wt D_N}_x-\delta\biggr\}. 
	\end{equation}
	By the ``gap estimate'' in~\cite[Theorem~9.2]{B-notes}, we have
	\begin{equation}
		\label{E:6.38}
		\lim_{N\to\infty}\bbP(F_{N}^\cc)=0
	\end{equation}
	while
	\begin{equation}
		\label{E:6.39}
		\lim_{N\to\infty}\bbP(L_N>\delta)=0
	\end{equation}
	by the coupling to the continuum ``binding'' field in Lemma~\ref{lemma-3.5}. 
	
	On~$F_N$  and for sufficiently large $N$, for each $x\in\wt D^{2a}_N$ with~$h^{\wt D_N}_x\ge m_N-\wt s$  there exists a.s.-unique $y_x\in \Lambda_{r_N/2}(x)$ where $h^{\wt D_N}$ achieves its maximum on~$\Lambda_{r_N}(y_x)$. On $F_N\cap\{L_N\le\delta\}$,  and if $x/N$ is also a point of
	$\tilde\eta^D_N(\cdot\times(-s,\infty)\times\bbR^{\bbZ^2})$, we have
	\begin{equation}
		h^{D_N}_x = h^{\wt D_N}_x + \varphi^N_x \le
		h^{\wt D_N}_{y_x} + \varphi^N_{y_x} + \delta =
		h^{D_N}_{y_x} + \delta \le h^{D_N}_x + \delta
	\end{equation}
	and thus $|h_x^{\wt D_N}-h^{\wt D_N}_{y_x}-\varphi^N_{y_x}|\le\delta$.
	For $\wt s \geq s+ \max_{\wt D^a_N}|\varphi^N|$, we have $h^{\wt D_N}_x\geq m_N - \wt s$
	whenever $x/N$ is a point of $\tilde\eta^D_N(\cdot\times(-s,\infty)\times\bbR^{\bbZ^2})$.
	Denoting $\text{osc}_u( f):=\sup\{| f(y)- f(x)|\colon |y-x|<u\}$ and assuming~$f\ge0$, we then obtain
	\begin{equation}
		\label{e:etaGM}
		\bigl\la  \tilde\eta^D_N,  \, f \otimes \Ind_{(-s,\,\infty)}\otimes 1 \bigr\ra
		\le \text{osc}_{r_N/N}( f)\,  \tilde\eta^D_N  \bigl(D\times(-s,\infty)\times  \R^{\bbZ^2}  \bigr)
		+\bigl\la \tilde \eta^{\wt D}_N,\,  f \otimes \Ind_{(-s-2\delta-\Phi^{D,\wt D},\,\infty)}\otimes 1 \bigr\ra,
		\quad
	\end{equation}
	where we also used Lemma~\ref{lemma-3.5} to couple $\varphi^N$ to a continuum ``binding'' field such that  also $E_N:=\bigl\{\max_{x\in\wt D^a}\bigl| \varphi^N_{\lfloor xN\rfloor}-\Phi^{D,\wt D}_x\bigr|\leq \delta\bigr\}$ occurs, besides $F_N\cap\{L_N\le\delta\}$. 
	
	We now take~$N\to\infty$  in~\eqref{e:etaGM}  along the subsequence such that (taking further subsequence if necessary) the pair $(\tilde\eta^D_N, \tilde\eta^{\wt D}_N)$ converges to $(\tilde\eta^D,\tilde\eta^{\wt D})$ in law. Relying on \twoeqref{E:6.38}{E:6.39} and also Lemma~\ref{lemma-3.5},  we  eliminate the contribution of $F_N^\cc\cup E_N^\cc\cup\{L_N>\delta\}$.
	 Using furthermore that $\max_{\wt D^a}|\Phi^{D,\wt D}|$ is a.\,s.\ finite and that  the limit measure~$\tilde\eta^{\wt D}$ is continuous in the second coordinate we then take also~$\delta\downarrow0$  and $\wt s\to\infty$  in~\eqref{e:etaGM},  and we  get the pointwise comparison
	\begin{equation}
		\bigl\la \tilde \eta^D,\,  f  \otimes \Ind_{(-s,\,\infty)}\otimes 1 \bigr\ra
		\le \bigl\la \tilde \eta^{\wt D},\, f \otimes \Ind_{(-s-\Phi^{D,\wt D},\,\infty)}\otimes 1 \bigr\ra \quad\text{a.s.}
	\end{equation}
	A completely analogous argument, which we skip for brevity, shows that ``$\ge$'' holds as well.
	Combining with~\eqref{e:zetaN-GM}, we conclude 
	\begin{equation}
	\begin{aligned}
		\Bigl(\tilde\zeta^D(\rmd x),\tilde\eta^D\bigl(\rmd \tilde x\times
		(-s,\,\infty)&\times \R^{\Z^2}\bigr)\Bigr)_{s\in\Q_+} \\
		&\laweq
		\Bigl(\rme^{\alpha\Phi^{D,\wt D}(x)}\tilde\zeta^{\wt D}(\rmd x),
		\tilde\eta^{ \wt D}(\rmd \tilde x\times \bigl(-s-\Phi^{D,\wt D}(\tilde x),\,\infty\bigr)\times \R^{\Z^2}\bigr)\Bigr)_{s\in\Q_+},
	\end{aligned}
	\end{equation}
	where $\Phi^{D,\wt D}$ is independent of $(\tilde\zeta^{\wt D},\tilde\eta^{\wt D})$ on the right-hand side. 
	The assertion now follows from Lemma~\ref{lemma-lln}.
\end{proof}

We are ready to give:

\begin{proof}[Proof of Theorem~\ref{thm-5.6}]
	 Let $(\zeta^D,\tilde\eta^D)$ be a weak subsequential limit of measures $(\zeta_N^D,\tilde\eta_N^D)$ and define $(\tilde\zeta^D,\vartheta^D)$ by  $\tilde \zeta^D(\cdot) := \zeta^D(\cdot \times\bbR)$ and defining $\vartheta^D$ from $\tilde\eta^D$ is as in~\eqref{E:6.30i}. Since the law of~$\vartheta^D$ is known from Lemma~\ref{lemma-lln},~$\vartheta^D$ satisfies all of the conditions of Theorem~\ref{BL2-thm2.8} with~$c:=\bar c$ in \eqref{E:1.25a}. The same  follows  with $c:=gc_\star$  for~$\tilde\zeta^D$  from  the proof of Theorem~\ref{thm-5.5}.
	 Indeed, for any $f\in C(\overline D)$ and $f_n\in C_\cc(\bbR)$ with $f_n \uparrow 1$, the Dominated Convergence Theorem gives
	\begin{equation}\label{e:cLQG-zeta}
		(g c_\star)^{-1} \langle \tilde\zeta^D , f\rangle 
		= (g c_\star)^{-1} \lim_{n\to\infty}  \langle \zeta^D , f\otimes f_n \rangle
		= \Bigl((g c_\star)^{-1}\lim_{n\to\infty}c_\star\int_0^\infty f_n(t) t\texte^{-\frac{t^2}{2g}}\textd t  \Bigr)\Bigl(\lim_{n\to\infty} \langle W^D_{f_n} ,f\rangle\Bigr) 
	\end{equation}
	where $W^D_{f_n}$ is from~\eqref{E:5.25} and has the law of the cLQG. The  expression in the first large parentheses on the right equals~one.

	The joint validity of the Gibbs-Markov property  for $\vartheta^D$ and $\tilde\zeta^D$  holds true by  Lemma~\ref{lemma-GM2}.  
	Theorem~\ref{thm-3.10} thus implies
	\begin{equation}
		\label{E:5.55}
		\bar c^{-1}\vartheta^D =  (gc_\star)^{-1}  \tilde \zeta^D\quad\text{\rm\ a.s.}
	\end{equation}
	Denote $Z^D:=(gc_\star)^{-1}\tilde \zeta^D$ and note that, by Lemma~\ref{lemma-lln}, this measure has the law of the cLQG. Theorem~\ref{thm-5.5} then shows that~$\zeta^D$ is determined completely by the sample of~$Z^D$ and, in fact,
	\begin{equation}
		\zeta^D(\textd x\textd t)=c_\star Z^D(\textd x)\otimes \Ind_{(0,\infty)}(t)t\texte^{-\frac{t^2}{2g}}\textd t.
	\end{equation}
	Lemma~\ref{lemma-lln}  and~\eqref{E:5.55} in turn give  \eqref{e:etaPPP}, with the Poisson process conditionally independent of~$Z^D$.  Hence, the subsequential limit $(\zeta^D, \eta^D)$ obeys~\eqref{E:5.30}.  The limit object does not depend on the chosen subsequence, so we also get full convergence. 
\end{proof}

\subsection{Including the field itself}
As our last item left to do, we need to include the DGFF itself in the distributional convergence.  This is based on the following observation: 

\begin{lemma}
	\label{lemma-6.10}
	The  weak limit
	\begin{equation}
		\label{E:2.10ue}
		\bigl(\zeta^D_N,\tilde\eta^D_N,h^{D_N}\bigr)\,\,\underset{N \to \infty}{\Lawarrow}\,\, (\zeta^D,\tilde\eta^D,h^D)
	\end{equation}
	exists with the law of the right-hand side determined by
	\begin{equation}
	\begin{aligned}
		\label{E:5.60b}
		E\bigl(\texte^{-\langle\zeta^D,f_1\rangle-\langle\tilde\eta^D, f_2\rangle+h^D(\rho)}\bigr)
		=&\,\texte^{\frac12\langle \rho,\wh G^D\rho \rangle}\,E\biggl(\exp\Bigl\{-c_\star\int_D Z^D(\textd x)\texte^{\alpha (\wh G^D\rho)(x)}\int_0^\infty\textd t\, t\texte^{-\frac{t^2}{2g}}\,f_1(x,t)
		\\
		&-\bar c\int_D Z^D(\textd x)\texte^{\alpha (\wh G^D\rho)(x)}\int_{\R}\textd h\,\texte^{-\alpha h}\int_{\R^{\Z^2}}\nu(\textd\phi)\bigl(1-\texte^{- f_2(x,h,\phi)}\bigr)\Bigr\}\biggr),
	\end{aligned}
	\end{equation}
	for each non-negative~$f_1\in C_\cc(D\times\R)$ and $f_2\in C_\cc(D\times\R\times\R^{\Z^2})$ and arbitrary~$ \rho\in \mathcal{M}_\cc(D)$.
\end{lemma}

\begin{proof}
	Under the measure 
	\begin{equation}
		A\mapsto\BbbP\bigl(\texte^{h^{D_N}(\rho)-\frac12\langle \rho, G^{D_N}\rho\rangle}\Ind_A\bigr)
	\end{equation}
	the DGFF $h^{D_N}$ has the law of $h^{D_N}+\rho G^{D_N}$ under~$\BbbP$.  By Lemma~\ref{lemma-3.2}, $G^{D_N}\rho$,  with the argument scaled by~$N$, tends uniformly to~$\wh G^D\rho$ as~$N\to\infty$.  Hence, we get
	\begin{equation}
		\E\bigl(\texte^{-\langle\zeta^D_N,f_1\rangle-\langle\tilde\eta^D_N, f_2\rangle+h^{D_N}(\rho)}\bigr)
		=o(1)+\texte^{\frac12\langle \rho,\wh G^D\rho\rangle}\E\Bigl(\texte^{-\langle\zeta^D_N,\,\texte^{\alpha\wh G^D\rho}f_1\rangle-\langle\tilde\eta^D_N, f_2\circ\theta_{\wh G^D\rho}\rangle}\Bigr),
	\end{equation}
	where $(f_2\circ\theta_s)(x,h,\phi) := f_2(x,  h(x)+s(x) ,\phi)$ and $o(1)\to0$ as~$N\to\infty$. The joint convergence of $(\zeta^D_N,\tilde\eta^D_N)$ to~$(\zeta^D,\tilde\eta^D)$ then shows
	\begin{equation}
		E\bigl(\texte^{-\langle\zeta^D_N,f_1\rangle-\langle\tilde\eta^D_N, f_2\rangle+h^{D_N}(\rho)}\bigr)
		\,\,\underset{N\to\infty}\longrightarrow\,\,
		\texte^{\frac12\langle \rho,\wh G^D\rho\rangle}E\Bigl(\texte^{-\langle\zeta^D,\,\texte^{\alpha\wh G^D\rho}f_1\rangle-\langle\tilde\eta^D, f_2\circ\theta_{\wh G^D\rho}\rangle}\Bigr).
	\end{equation}
	This proves the existence of the joint limit in \eqref{E:2.10ue}. Invoking \eqref{E:5.30} along with a routine change of variables, we then get \eqref{E:5.60b} as well.
\end{proof}

We are finally ready to give:

\begin{proof}[Proof of Theorem~\ref{thm-A}]
	Lemma~\ref{lemma-6.10} shows the 
	 weak convergence of $(\zeta^D_N,\tilde\eta^D_N,h^{D_N})$ as~$N\to\infty$. Moreover, \eqref{E:5.60b} with $f_2:=0$ and $f_1(x,t):=c g(x)$ for a suitable constant~$c>0$ shows that the law of~$Z^D$ underlying the limit processes~$\zeta^D$ and~$\tilde\eta^D$ obeys
	\begin{equation}
		E\bigl(\texte^{-\langle Z^D,g\rangle+h^D(\rho)}\bigr)
		=\texte^{\frac12\langle\rho,\wh G^D\rho\rangle}\E\bigl(\texte^{-\langle \texte^{\alpha\wh G^D\rho}Z^D,g\rangle}\bigr)
	\end{equation}
	for all~$\rho\in \mathcal{M}_\cc(D)$  and all non-negative~$g\in C_\cc(D)$. Theorem~\ref{thm-4.3} implies that~$Z^D$ is, up to a modification on a set of vanishing probability, the cLQG associated with~$h^D$. 
	
	This proves the claim for the structured extremal process; the reduction to the unstructured process is then the same as in the proof of \cite[Corollary~2.2]{BL3}.
\end{proof}

\begin{proof}[Proof of Corollary~\ref{cor-2.5}]
	For $M>0$ arbitrary, we first show the assertion for the boundedly supported continuous function $\Phi_M$ in place of $\Phi$, where $\Phi_M$ is defined by $\Phi_M(t)=\Phi(t)$ for $t\in[-1,M]$, $\Phi_M(t)=0$ for $t\in(-\infty,-2]\cup[M+1,\infty)$, and the linear interpolation in between.
	Write~$\vartheta^D_{N,M}$ for the measure on the left of \eqref{E:2.14} with $\Phi_M$ in place of $\Phi$, and let~$f\colon D\to\R$ be continuous  with compact support in~$D$.  Then $\langle\vartheta^D_{N,M},f\rangle=\langle\zeta^D_N,f\otimes \Phi_M\rangle$. Since $f\otimes\Phi_M$ is continuous and compactly supported in $ D \times \bbR$, Theorem~\ref{thm-A} ensures that $\langle\zeta^D_N,f\otimes \Phi_M\rangle$ tends in law to $\langle\zeta^D,f\otimes \Phi_M\rangle=c(\Phi_M)\langle Z^D,f\rangle$ (where $c(\Phi_M)$ is defined in the same way as $c(\Phi)$ with $\Phi_M$ in place of $\Phi$). It follows that $\vartheta^D_{N,M}$ tends in law to~$c(\Phi_M)Z^D$ relative to the vague topology on $D$.
	
	Next, we write $\vartheta^D_N$ for the measure on the left-hand side of~\eqref{E:2.14}, and 
	for arbitrary $\delta>0$, $u\ge 1$ and sufficiently large~$N$, we estimate
	\begin{equation}
	\begin{aligned}
		\label{E:6.53}
		\BbbP\biggl( \Bigl| &\langle\vartheta^D_{N,M},f\rangle -  
		\langle\vartheta^D_N,f\rangle \Bigr| > \delta \biggr)
		\le \BbbP\bigl( \max h^{D_N} > m_N + u\bigr)
		\\
		&+\delta^{-1}\bbE\Biggl( \frac{\|f\|_\infty}{\log N}\sum_{x\in D_N}
		\Bigl|\Phi\Bigl(\frac{m_N-h^{D_N}_x}{\sqrt{\log N}}\Bigr)\Bigr|
		\Ind_{\{m_N-h^{D_N}_x\geq M\sqrt{\log N}\}}\,
		\texte^{\alpha(h^{D_N}_x-m_N)} 
		; h^{D_N} \le m_N + u \Biggr)
	\end{aligned}
	\end{equation}
	using a union bound and Markov inequality. We now choose $u$ such that the first term on the right hand side is bounded by $\delta$ for all $N> \rme^{u^2}$ by Lemma~\ref{lemma-DGFF-tail}. As $f$ has compact support in $D$, there exists $\delta'>0$ such that only $x\in D^{\delta'}_N$ contribute to the sum in the second term of~\eqref{E:6.53}. 
	Hence, the second term on the right-hand side of~\eqref{E:6.53} is for arbitrary $\ep>0$  bounded by a constant times
	\begin{equation}
		\delta^{-1}(1+u)\|f\|_\infty\int_{s\ge M} \rmd s \,\big| \Phi(s)\big| \rme^{-s^2\big(\tfrac{1}{2g}-\ep\big)},
	\end{equation}
	uniformly in all sufficiently large $N$, which follows as in the proof of Proposition~\ref{lem:first-moment}, analogously to~\eqref{e:5.24}. The integral in the last display vanishes as $M\to\infty$. 
	Hence, the left-hand side of~\eqref{E:6.53} is bounded by $2\delta$ for all sufficiently large~$N$ and~$M$.
	Moreover, from Dominated Convergence, we get that $[c(\Phi_M)-c(\Phi)]Z^Df$ vanishes as $M\to\infty$. By combining the estimates, we obtain that
	$\langle \vartheta^D_N,f\rangle$ tends in law to $\langle c(\Phi)Z^D,f\rangle$,
	which yields the assertion.
\end{proof}

%
%

\begin{acks}[Acknowledgments]
 We  are grateful  to an anonymous referee for valuable comments.
\end{acks}
\begin{funding}
This project has been supported in part by the NSF awards DMS-1712632 and DMS-1954343, ISF grants No.~1382/17 and~2870/21 and BSF award 2018330. The second author has also been supported in part by a Zeff Fellowship  at the Technion, by a Minerva Fellowship of the Minerva  Stiftung  Gesellschaft f\"ur die Forschung mbH and by the DFG (German Research Foundation) grant No.~2337/1-1 (project 432176920).
\end{funding}



\begin{thebibliography}{XX}
	
	\bibitem{AB}
	Y. Abe and M. Biskup (2022). Exceptional points of two-dimensional random walks at multiples of the cover time.  \textit{Probab. Theory Rel. Fields} \textbf{183}, 1--55.
	
	\bibitem{ABL}
	Y. Abe, M. Biskup and S. Lee (2023). Exceptional points of discrete-time random walks in planar domains.  \textit{Electron. J. Probab.} (to appear).  arXiv:1911.11810
	
	
	\bibitem{AHPS}
	 J.~Aru, N.~Holden, E.~Powell and X.~Sun (2023). Brownian half-plane excursion and critical Liouville quantum gravity.
	\textit{J. London Math. Soc. (2)} \textbf{107}, 441--509.
	
	

	\bibitem{Berestycki}
	N.~Berestycki (2017).
	An elementary approach to Gaussian multiplicative chaos.
	\textit{Electron. Commun. Probab.} \textbf{22}, 1--12.
	

	
	\bibitem{B-notes}
	M. Biskup (2020). 
	Extrema of the two-dimensional Discrete Gaussian Free Field. In: M.~Barlow and G.~Slade (eds.): Random Graphs, Phase Transitions, and the Gaussian Free Field. SSPROB 2017. Springer Proceedings in Mathematics \&\ Statistics, vol 304, pp 163--407. Springer,~Cham. 
	
	\bibitem{BGL}
	M.~Biskup, S.~Gufler and O.~Louidor (2023).   On support sets  of the critical Liouville Quantum Gravity. In preparation.
	
	\bibitem{BL1} 
	M. Biskup and O. Louidor (2016).
	Extreme local extrema of two-dimensional discrete Gaussian free field.
	\textit{Commun. Math. Phys.} \textbf{345} 271-304.
	
	\bibitem{BL2}
	M. Biskup and O. Louidor (2020).
	Conformal symmetries in the extremal process of two-dimensional discrete Gaussian free field.
	\textit{Commun. Math. Phys.} \textbf{375}, no.~1, 175--235.
	
	\bibitem{BL3} 
	M. Biskup and O. Louidor (2018).
	Full extremal process, cluster law and freezing for two-dimensional discrete Gaussian free field.
	\textit{Adv. Math.} \textbf{330}  589-687.
	
	\bibitem{BL4} 
	M. Biskup and O. Louidor (2019).
	On intermediate level sets of two-dimensional discrete Gaussian free field. \textit{Ann. Inst. Henri Poincar\'e} \textbf{55}, no.~4, 1948--1987.
	
	
	\bibitem{BDingZ}
	M. Bramson, J. Ding and O. Zeitouni (2016). 
	Convergence in law of the maximum of the two-dimensional discrete Gaussian free field. 
	\textit{Commun. Pure Appl. Math} \textbf{69}, no.~1, 62--123.
	
	\bibitem{BZ}
	M. Bramson and O. Zeitouni (2012). 
	Tightness of the recentered maximum of the two-dimensional discrete Gaussian free field. 
	\textit{Comm. Pure Appl. Math.} \textbf{65}, 1--20.
	
	
	
	\bibitem{DZ}
	J. Ding and O. Zeitouni (2014). 
	Extreme values for two-dimensional discrete Gaussian free field. 
	\textit{Ann. Probab.} \textbf{42}, no.~4, 1480--1515
	
	\bibitem{DRSV1}
	B. Duplantier, R.~Rhodes, S.~Sheffield and V.~Vargas  (2014).
	Critical Gaussian multiplicative chaos: Convergence of the derivative martingale. 
	\textit{Ann. Probab.} \textbf{42}, no.~5, 1769--1808.
	
	\bibitem{DRSV2}
	B. Duplantier, R.~Rhodes, S.~Sheffield and V.~Vargas (2014).
	Renormalization of critical Gaussian multiplicative chaos and KPZ formula.  
	\textit{Commun. Math. Phys.} \textbf{330}, no.~1,  283--330.
	
	\bibitem{DS} 
	B. Duplantier and S. Sheffield (2011).
	Liouville quantum gravity and KPZ.
	\textit{Invent. Math.} \textbf{185}, 333-393.
	
	\bibitem{DSM21}  B. Duplantier, J.~Miller and S.~Sheffield (2021).
	Liouville quantum gravity as a mating of trees.
	\textit{Astérisque} no.~427. 
	
	
	\bibitem{Ballot}
	S.~Gufler and O.~Louidor (2022). Ballot theorems for the two-dimensional discrete Gaussian free field. \textit{J. Stat. Phys.} \textbf{189}, no.~13, 1--98. 
	
	
	\bibitem{Jego}
	A. Jego (2023). Characterisation of planar Brownian multiplicative chaos. \textit{Commun. Math. Phys.} {\bf 399}, 971 -- 1019.
	
	
	\bibitem{JS}
	J. Junnila and E. Saksman (2017). 
	Uniqueness of critical Gaussian chaos. \textit{Elect. J.~Probab} \textbf{22}, 1--31.
	
	\bibitem{Kahane}
	J.-P. Kahane  (1985). 
	Sur le chaos multiplicatif. 
	\textit{Ann. Sci. Math. Qu\'ebec} \textbf{9}, no.2, 105--150.
	
	
	\bibitem{Lawler-Limic}
	G.F.~Lawler and V.~Limic (2010).
	\textit{Random walk: a modern introduction}. 
	Cambridge Studies in Advanced Mathematics, vol.~123. Cambridge University Press, Cambridge, xii+364.
	
	
	
	\bibitem{RV-review}
	R.~Rhodes and V.~Vargas (2014).
	Gaussian multiplicative chaos and applications: A review.
	\textit{Probab. Surveys} \textbf{11}, 315--392
	
	\bibitem{Powell}
	E.~Powell (2018).
	Critical Gaussian chaos: convergence and uniqueness in the derivative normalisation.
	\textit{Electron. J. Probab.} \textbf{23}, paper no.~31, 26~pp.
	
	
	\bibitem{Shamov}
	A.~Shamov (2016). On Gaussian multiplicative chaos. \textit{J. Funct. Anal.} \textbf{270}, no.~9, 3224--3261
	
\end{thebibliography}

\bibliographystyle{abbrv}

\end{document}